\documentclass{amsart}
\usepackage{amssymb}
\usepackage{amsmath}
\usepackage{hyperref}
\usepackage{color}
\usepackage{graphics}
\usepackage{graphicx}
\usepackage{texdraw}
\newcommand{\what}{\widehat}%
\newcommand{\wtilde}{\widetilde}%
\newcommand{\R}{\mathbb R}%
\newcommand{\C}{\mathbb C}%
\newcommand{\Z}{\mathbb Z}%
\newcommand{\hc}{\mathrm c}
\newtheorem{theorem}{Theorem}[section]
\newtheorem{thmspecial}{Theorem}

\newtheorem{lemma}[theorem]{Lemma}
\newtheorem{proposition}[theorem]{Proposition}
\newtheorem{corollary}[theorem]{Corollary}

\theoremstyle{definition}

\theoremstyle{definition}
\newtheorem{remark}[theorem]{Remark}

\numberwithin{equation}{subsection}
\numberwithin{theorem}{subsection}

\begin{document}
\baselineskip14pt

\author[P. Kumar]{Pratyoosh Kumar}
\address[P. Kumar]{Department of Mathematics and Statistics, Indian Institute of Technology Kanpur,
Kanpur 208016,  India, E-mail: prkumar@iitk.ac.in}

\author[S. K. Ray]{Swagato K. Ray }
\address[S. K. Ray]{Department of Mathematics and Statistics, Indian Institute of Technology Kanpur,
Kanpur 208016,  India, E-mail: skray@iitk.ac.in}

\author[R. P. Sarkar]{Rudra P. Sarkar}
\address[R. P. Sarkar]{Stat-Math Unit, Indian Statistical
Institute, 203 B. T. Rd., Calcutta 700108, India, email:
rudra@isical.ac.in}

\title[Eigenfunction of the Laplacian]{Characterization of almost  $L^p$-eigenfunctions of the
Laplace-Beltrami operator}

\begin{abstract} In \cite{Roe} Roe proved that
if  a doubly-infinite sequence $\{f_k\}$ of functions on $\R$ satisfies $f_{k+1}=(df_{k}/dx)$ and  $|f_{k}(x)|\leq M$ for all $k=0,\pm 1,\pm 2,\cdots$ and $x\in \R$,
then $f_0(x)=a\sin(x+\varphi)$ where  $a$ and $\varphi$ are real constants.
This result was  extended to $\R^n$
by
Strichartz \cite{Str} where $d/dx$ is substituted by  the Laplacian on $\R^n$.
While it is
plausible to extend this theorem for other Riemannian manifolds or Lie groups,
Strichartz showed that the result holds true for Heisenberg groups,
but  fails for hyperbolic $3$-space. This negative result can be
indeed extended to any Riemannian symmetric space of noncompact type.
We observe that this failure is rooted in the
$p$-dependance of the $L^p$-spectrum of the Laplacian on the
hyperbolic spaces. Taking this into account we shall prove that for
all rank one Riemannian symmetric spaces of noncompact type, or more
generally for the harmonic $NA$ groups, the theorem actually holds
true when uniform boundedness is replaced by uniform ``almost $L^p$
boundedness''.
In addition we shall see that for the
symmetric spaces this theorem is capable of characterizing  the Poisson
transforms of  $L^p$ functions on the boundary, which some what
resembles the original theorem of Roe on $\R$.
\end{abstract}
\maketitle

\setcounter{tocdepth}{1}
\tableofcontents

\section{Introduction}
This paper  revolves mainly around results characterizing
eigenfunctions of the Laplace-Beltrami operator $\Delta$ on
Riemannian symmetric spaces  of noncompact type with real rank
one (which we shall  denote by $X$) and its nonsymmetric generalizations namely the Damek-Ricci (DR)
spaces (which will be denoted by $S$)  in which the
former spaces account for a very thin sub class (see \cite{ADY}). DR spaces are  also known as harmonic $NA$ groups. They are solvable Lie groups as well as harmonic manifolds and appear as counter examples  to the Lichnerowicz conjecture  in the noncompact case (see \cite{Damek-Ricci}).

We are concerned about the following generalization of
Roe's theorem  proved by  Strichartz  (\cite{Str}) which involves the Laplace
 operator $\Delta_{\R^n}$ of $\R^n$. (See also \cite{Howd-1, Howd-Reese, Kim-Chung} and the references therein.)
 \begin{theorem}[Strichartz]
Let $\{f_k\}_{k\in\Z}$ be a doubly infinite sequence of functions on
$\R^n$ with $\|f_k\|_{L^{\infty}(\R^n)}\leq M$ for all $k\in\Z.$ If
for some $\alpha>0$,  $\Delta_{\R^n} f_k=\alpha f_{k+1}$ for all
$k\in\Z$, then $\Delta_{\R^n} f_0=-\alpha f_0.$\label{bob}
\end{theorem}
 The case $\alpha=1$ was proved in \cite{Str}, but the same proof
works for other values of $\alpha$ as well. It is not difficult to observe that
the  theorem above holds true if one replaces
$L^{\infty}(\R^n)$ by $L^p(\R^n)$-norm with the restriction $p>2n/(n-1)$ for the result to be non-vacuous (see
\cite{AN}) or by weak $L^p$-norm for $p\ge 2n/(n-1)$ (which can be substantiated by
standard estimate of Bessel functions (see \cite{S-W})). Our starting point  however is  a striking counter example in \cite{Str}
 which shows that the  result above is no
longer true if $\R^n$ is replaced by the  symmetric space
$SL(2,\C)/SU(2)$. Precisely, there exists a uniformly bounded doubly
infinite sequence $\{f_k\}_{k\in\Z}$ of radial eigenfunctions of
$\Delta$ on $SL(2,\C)/SU(2)$ satisfying $\Delta f_k=f_{k+1}$ but
$\Delta f_0\neq -f_0.$ This counter example can be strengthened. Precisely, in any $X$ or $S$ a sequence  $\{f_k\}$ can be constructed which satisfies the hypothesis of the theorem above with uniformly bounded (or uniformly bounded with respect to  $L^p$-norm, $2<p<\infty$), but   $f_0$
is not even an eigenfunction of $\Delta$ (see section 3).  This motivated us to have
a detailed investigation of the phenomenon in the context of symmetric or DR spaces.
 A somewhat deeper understanding of the
counter example mentioned above tells us that the failure of the
result for hyperbolic spaces can be ascribed to the, by now well
known, fact that the $L^p$-spectrum of the Laplacian on $X$ or on $S$ depends
on $p$ (see \cite{Lo-Ry-2, Tay, ADY}). This is one of the
most intriguing difference between analysis on Lie groups with
polynomial growth and that of with exponential growth (see, for
instance \cite{Hu2, Hu3}).

Among other things in the present paper we shall obtain an
essentially sharp version of Theorem \ref{bob} on symmetric and on  DR spaces
which will involve various uniform size-estimates close to $L^p$,
instead of uniform boundedness.  These
size-estimates arise naturally due to the behavior of Poisson
transforms (of $L^p$-functions on the boundary) whose eigenvalues lie on the
boundary of the $L^p$-spectrum of $\Delta$ (see \cite{Lo-Ry, KRS,
Ray-Sarkar} and section 3.1) and among them weak
$L^p$-norm can be singled out by its translation invariance.  We mention  here two
representative theorems (for notation see section 2). We shall prove these results as a
consequence of a general version of Theorem \ref{bob} on DR spaces
$S$ in section 5.

\begin{thmspecial}\label{thm-X-2}
Let $f$ be a   measurable
functions on $S$  and $\alpha$ a nonzero real number. If   $\|\Delta^k f\|_{2, \infty}\le M(\alpha^2+\rho^2)^k$ for all $k\in \Z$, for some $M>0$,
 then $\Delta f=-(\alpha^2+\rho^2)f$. In particular when $S=X$ is an Iwasawa $NA$ group, then
  $f=\mathcal P_\alpha F$ for some $F\in L^2(K/M)$.

 \end{thmspecial}

\begin{thmspecial}\label{thm-X-p}
Let $f$ be a   measurable
functions on $S$ and $p\in (1,2)$.
If   $\|\Delta^k f\|_{p', \infty}\le M (4\rho^2/pp')^k=M ((i\gamma_{p'}\rho)^2+\rho^2)^k$ for $k=0,1,2,\dots$, for some $M>0$,
then $\Delta f=-(4\rho^2/pp')f$. In particular when $S=X$ is an Iwasawa $NA$ group, then
 $f=\mathcal P_{-i\gamma_p\rho} F$ for some $F\in L^{p'}(K/M)$.
 \end{thmspecial}
If we define $f_k=(4\rho/pp')^{-k}\Delta^kf$ then the  statements of these theorems  resemble Theorem \ref{bob}.
  We observe en passant that Theorem A and B have structural resemblance with the following celebrated  result of  Kotake and Narasimhan \cite{KoNar}:
\begin{theorem} Let $\Omega$ be an open set of $\R^n$. Let  $A$ be  a
 linear elliptic operator of order $m$ with analytic coefficients in $\Omega$. If a  $C^\infty$ function $f$ satisfies  for some $c>0$,
$\|A^k f\|_{L^2(\Omega)}\le (mk)! c^{k+1}$  for all nonnegative integers $k$, then $f$ is real analytic on $\Omega$.
\end{theorem}

It is not difficult to see that using the $G$-invariance of $\Delta$
on $X=G/K$, it is enough to restrict to the $K$-isotypic components of
$f$. However the lack of {\em rotation} on the DR spaces makes it more interesting and forces
us to adopt a different approach.

 We may stress that in the theorems above, for the symmetric
spaces, a concrete description of the eigenfunction $f$ as Poisson
transforms (of $L^p$-functions on the boundary) is achieved. A
crucial ingredient for this, in the case of Theorem \ref{thm-X-p} is
a characterization of eigenfunctions due to   Lohou\'{e} and Rychener
(\cite{Lo-Ry}, see also \cite{Sj}). The corresponding result
required for Theorem \ref{thm-X-2} seems to be new and will be
proved in section 4 using a result of Ionescu (\cite{Ion-Pois-1}).
A rich body of literature concerning representation
theorems of eigenfunctions of $\Delta$ on $X$ (see e.g. \cite{F, KW,
Stoll, Sj, Bou-Sami, BOS}, see also \cite{Dam1}) is already available.  In section 4 we shall briefly survey the existing results in this direction and  further generalize them, keeping our need in view. In particular we shall settle a question posed in \cite{Bou-Sami}. This section  may be of independent interest and is independent of the rest of the paper.   Some of these results  will be used to obtain analogues of Theorems \ref{thm-X-2} and Theorem \ref{thm-X-p}.

The phrase {\em almost $L^p$} is used in this paper to mean
the size estimates which are close to $L^p$-norm, weak $L^p$ being one of the  examples. In section 4 we shall obtain others. We shall use these estimates to formulate various analogues of Theorem \ref{thm-X-2} and Theorem \ref{thm-X-p}. For motivation we  cite a version of Theorem \ref{bob} on $\R^n$ which involves the following {\em almost $L^2$}-size estimate: \[M_2(f)=\left(\limsup_{R\to\infty}\frac{1}{R}\int_{B(0,R)}|f(x)|^2dx\right)^{1/2},\]
where $B(0, R)$ is the ball of radius $R$ in $\R^n$, centered at origin.
\begin{theorem} Let $\{f_k\}_{k\in \Z}$ be a
doubly infinite sequence of measurable functions on $\R^n$
satisfying  $\Delta f_k=\alpha^2 f_{k+1}$ for some $\alpha\in \R$ and
$M_2(f_k)\leq M$ for all $k\in\Z$. Then there exists  $F\in
L^2(S^{n-1})$ such that $f_0(x)=\int_{S^{n-1}}e^{i\alpha\langle
x,\omega\rangle} F(\omega)d\omega$.
\end{theorem}
The result above is an easy consequence of Lemma 3.2 in
\cite{St1} and the idea of the proof of Theorem \ref{bob}. The details is left to the interested readers.
 As mentioned earlier,  unlike uniform boundedness, the size estimate used in the
 formulation above is not translation invariant. An analogue of this result for symmetric spaces and for more general $p$ will be proved in section 6. We note here  that on DR spaces, concrete realization  of a function (satisfying the hypothesis of Theorem A or Theorem B) as Poisson transform of a function on its boundary seems
to be more involved and is still open.

There are important eigenfunctions of $\Delta$, e.g. the (powers of) Poisson kernel: $e^{(i\alpha+\rho)H(x^{-1}k)}$, which are the objects analogous to the functions $x\mapsto e^{i\langle\lambda, x\rangle}$ on $\R^n$. However they do not belong to any $L^p$, weak-$L^p$ or in general in any Lorentz spaces. They are not even bounded, unlike their Euclidean counter parts. In particular they do not satisfy the hypotheses of Theorem \ref{thm-X-2} or Theorem \ref{thm-X-p}. One of the purposes of the formulations of the Theorems in section 5 is to prevent their a priori exclusions, where we shall  weaken the hypothesis suitably using $L^p$-tempered distributions.

\noindent{\bf Acknowledgements:} Authors are grateful  to S. C. Bagchi for some very valuable help.
We would  also like to thank  M. Cowling, D M\"{u}ller and J. Sengupta for some illuminating  conversations with them.

\section{Preliminaries} The preliminaries and notation related to the semisimple Lie groups and the
associated symmetric spaces are standard and can be found for example
in \cite{Helga-2}, while that related to DR spaces can be retrieved from  \cite{ACB, ADY, Ray-Sarkar, KRS}. To make the article self-contained  we shall gather only those  results which are required for this paper.
\subsection{Generalities}

For any $p\in [1, \infty)$, let $p'=p/(p-1)$. The letters $\Z$, $\R$, $\C$ and $\mathbb H$ denote respectively the set of integers, real numbers, complex numbers and quaternions.  For
$z\in \C$,  $\Re z$ and $\Im z$  denote respectively the real and
imaginary parts of $z$. We denote the nonzero real numbers and
nonnegative integers respectively by $\R^\times$ and $\Z^+$.  For a set $A$ in a measure space we shall use $|A|$ to denote the measure of $A$. We
shall follow the standard practice of using the letters $C, C_1, C_2$
etc. for positive constants, whose value may change from one line to another.
Occasionally the constants will be suffixed to show their
dependencies on important parameters. Everywhere in this article the
symbol $f_1\asymp f_2$ for two positive expressions $f_1$ and $f_2$
means that there are positive constants $C_1, C_2$ such that
$C_1f_1\leq f_2\leq C_2f_1$.

Apart from the Lebesgue spaces we also need to deal with the Lorentz
spaces which we shall introduce briefly (see \cite{Graf, S-W,
Ray-Sarkar} for details).
Let $(M, m)$ be a $\sigma$-finite measure space, $f:M\longrightarrow \C$ be a
measurable function and $p\in [1, \infty)$, $q\in [1, \infty]$. We
define
\begin{equation*}\|f\|^*_{p,q}=\begin{cases}\left(\frac qp\int_0^\infty [f^*(t)t^{1/p}]^q\frac{dt}t\right)^{1/q}
\ \ \ \ \ \ \ \ \  \ \ \ \ \ \textup{ if } q<\infty,\\ \\ \sup_{t>0}td_f(t)^{1/p}=\sup_{t>0} t^{1/p} f^\ast(t)
 \ \ \textup{ if } q=\infty,\end{cases}\end{equation*} where
for $\alpha>0$, $d_f(\alpha)=|\{x \mid f(x)>\alpha\}|$ is the distribution function of $f$ and
$f^*(t)=\inf\{s \mid d_f(s)\le t\}$ is the {\em
decreasing rearrangement} of $f$.
We take
$L^{p,q}(M)$ to be the set of all measurable $f:M\longrightarrow \C$ such that
$\|f\|^*_{p,q}<\infty$. For $1\le p <\infty$, $L^{p,p}(M)=L^p(M)$
and $\|\cdot\|_{p,p}^\ast=\|\cdot\|_p$. By $L^{\infty, \infty}(M)$
and $\|\cdot\|^\ast_{\infty, \infty}$ we mean respectively the space
$L^\infty(M)$ and the norm $\|\cdot\|_\infty$. The space
$L^{p,\infty}(M)$ is known as the weak $L^p$-space. Following
properties of the Lorentz spaces will be required:
\begin{enumerate}
\item[(i)] For $1<p, q<\infty$, the dual space of $L^{p,q}(M)$ is $L^{p', q'}(M)$ and the dual of $L^{p,1}(M)$ is $L^{p',\infty}(M)$.
 \item[(ii)] If $q_1\le q_2\le \infty$ then
     $L^{p, q_1}(M)\subset L^{p, q_2}(S)$ and $\|f\|^\ast_{p, q_2}\le \|f\|^\ast_{p, q_1}$.
\end{enumerate}
The Lorentz  ``norm'' $\|\,\cdot\,\|^\ast_{p,q}$ is indeed only
a quasi-norm  and this makes the space $L^{p,q}(M)$ a quasi Banach
space (see \cite[p. 50]{Graf}). However for $1<p\le \infty$, there
is an equivalent norm $\|\,\cdot\,\|_{p,q}$ which makes it a Banach
space (see \cite[Theorems 3.21, 3.22]{S-W}). We shall slur over this difference and use the notation $\|\cdot\|_{p,q}$.

\subsection{Damek-Ricci spaces}
Let $\mathfrak n=\mathfrak v\oplus \mathfrak z$ be a $H$-type
Lie algebra where $\mathfrak v$ and $\mathfrak z$ are vector spaces
over $\R$ of dimensions $m$ and $l$ respectively. Indeed
$\mathfrak z$ is the centre of $\mathfrak n$ and $\mathfrak v$ is
its ortho-complement with respect to the inner product of
$\mathfrak n$. Then we  know that $m$ is even.  The group law of $N=\exp \mathfrak n$ is given by
$$(X, Y). (X', Y')=((X+X', Y+Y'+\frac 12[X, X'])\ \ X\in \mathfrak v, Y\in \mathfrak z.$$
We shall identify
$\mathfrak v$ and $\mathfrak z$ and $N$ with $\R^m$, $\R^l$ and
$\R^m\times \R^l$ respectively.  The group  $A=\{a_t=e^t \mid t\in \R\}$  acts on $N$ by nonisotropic
dilation: $\delta_{t}(X,Y)=(e^{t/2}X, e^{t}Y)$.
Let $S=NA=\{(X, Y, a_t)\mid (X, Y)\in N, t\in \R\}$ be
the semidirect product of $N$ and $A$ under the action above. The group law of $S$ is thus given by:
$$(X, Y, a_t) (X', Y', a_s)=(X+a_{t/2} X', Y+ a_{t} Y'+ \frac  {a_{t/2}}2 [X, X'], a_{t+s}).$$ It then follows that
$\delta_{t}(X,Y)=a_tna_{-t}$, where $n=(X, Y)$.
The Lie group
$S$ is  solvable, connected and simply connected  with
Lie algebra $\mathfrak s=\mathfrak v\oplus\mathfrak z\oplus\R$. It
is well known that $S$ is  nonunimodular. The
homogenous dimension of $S$ is $Q=m/2+l$. For convenience we shall also use the notation $\rho=Q/2$.
The group $S$ is equipped with the left-invariant Riemannian
metric $d$ induced by
\begin{equation*}\langle(X, Z, \ell), (X', Z', \ell')\rangle =\langle X, X'\rangle+\langle Z, Z'\rangle+\ell\ell'
\end{equation*} on $\mathfrak s$.
The associated left invariant Haar measure $dx$ on $S$   is given
by
\begin{equation} \int_S f(x)dx=\int_{N\times A}f(na_t)e^{-Q
t}dtdn,
\label{measure-NA}
\end{equation} where $dn(X,Y)=dX\,dY$ and $dX, dY, dt$ are Lebesgue
measures on $\mathfrak v$, $\mathfrak z$ and $\R$ respectively.
We denote the Laplace-Beltrami operator associated to this Riemannian structure by $\Delta$.

The group $S$ can also be  realized as the unit ball
\begin{equation*}B(\mathfrak  s)=\{(X, Z, \ell)\in \mathfrak  s\mid |X|^2+|Z|^2+\ell^2<1\}\end{equation*}
via a Cayley transform $C:S\longrightarrow     B(\mathfrak s)$ (see
\cite[p.~646--647]{ADY} for details). For an element $x\in S$, let
$$|x|=d(C(x), 0)=d(x, e)=\log \frac{1+\|C(x)\|} {1-\|C(x)\|}.$$ In particular $d(a_t, e)=|t|$. The left Haar measure in geodesic polar
coordinates is  given by (\cite[(1.16)]{ADY})
\begin{equation}dx=2^{m}(\sinh r)^{k}(\sinh \frac r2)^m\,dr\, d\omega
\label{DR-polar}
\end{equation} where $r=|x|$ and  $d\omega$ denotes the
surface measure on the unit  sphere $\partial B(\mathfrak s)$ in
$\mathfrak  s$. For convenience we shall write the corresponding integral formula
as $\int_Sf(x)dx=\int_0^\infty \int_{\partial B(\mathfrak s)}f(rw) J(r) dr\, d\omega$.
A function $f$ on $S$ is called {\em radial} if for all $x,y\in
S$, $f(x)=f(y)$ if $d(x,e)=d(y,e)$. By abuse of notation we shall sometimes consider a radial function $f$ as a function of $|x|$ and for such a function
$\int_Sf(x)dx=\int_0^\infty f(r)J(r)dr$.
 For a function space $\mathcal L(S)$ on $S$ we denote its
subspace of radial functions by $\mathcal L(S)^\#$.
For a suitable function $f$ on $S$ its radialization $Rf$ is
defined as
\begin{equation}Rf(x)=\int_{S_\nu}f(y)d\sigma_\nu(y),
\label{radialization}
\end{equation} where $\nu=|x|$ and $d\sigma_\nu$ is the
surface measure induced by the left invariant Riemannian metric on
the geodesic sphere $S_\nu=\{y\in S\mid d(y, e)=\nu\}$ normalized
by $\int_{S_\nu}d\sigma_\nu(y)=1$. It is clear that $Rf$ is a
radial function and if $f$ is radial then $Rf=f$.
The following properties of the radialization operator will be needed (see \cite{DR, ACB}):
\begin{enumerate}
\item $\langle R \phi, \psi\rangle=\langle \phi, R\psi\rangle\ \ \phi, \psi\in C^\infty_c(S)$,
\item $R(\Delta f)=\Delta(R f)$.
\end{enumerate}
Since by (1) above, $\int_Sf(x)dx=\int_0^\infty R f(r)J(r)dr$, we have
$\|R f\|_1\le \|f\|_1$.
Interpolating (\cite[p. 197]{S-W}) with the trivial $L^\infty$-boundedness of $R$ we have,
$$\|R f\|_{q,r}\le \|f\|_{q,r}, 1<q<\infty, 1\le r\le \infty.$$

To proceed towards the Fourier transform we need to introduce the notion of Poisson kernel.
The Poisson kernel $\wp:S\times N\longrightarrow \R$ is given by
$\wp(x,n)=\wp(n_1a_t, n)=P_{a_t}(n^{-1}n_1)$ where
\begin{equation}P_{a_t}(n)=P_{a_t}(X, Y)=C a_t^{Q}\left(\left(a_t+\frac{|X|^2}4\right)^2+|Y|^2\right)^{-Q},\,\, n=(X,
Y)\in N. \label{poisson}\end{equation}  The  value of $C$ is
adjusted so that $\int_NP_a(n)dn=1$ (see
\cite[(2.6)]{ACB}). For $\lambda\in \C$, we define
$\wp_\lambda(x, n)=\wp(x,n)^{1/2-i\lambda/Q}=\wp(x,n)^{-(i\lambda-\rho)/Q}$. Then it is known that for each fixed $n\in N$
$\Delta \wp_\lambda(x, n)=-(\lambda^2+\rho^2) \wp_\lambda(x, n)$.
The Poisson transform of a function $F$ on $N$ is defined as (see
\cite{ACB})
\begin{equation*}{\mathfrak P}_{\lambda}F(x)=\int_N F(n)\wp_\lambda (x, n)dn.\end{equation*} It follows that
$\Delta\mathfrak P_\lambda F= -(\lambda^2+\rho^2) \mathfrak P_\lambda F$.
The elementary spherical function $\phi_\lambda(x)$ is  given by
$$\phi_\lambda(x)=
\int_N\wp_\lambda(x,n)\wp_{-\lambda}(e,n)dn.$$ It follows that
$\phi_\lambda$ is a radial eigenfunction of  $\Delta$  with eigenvalue $-(\lambda^2+\rho^2)$
satisfying $\phi_\lambda(x)=\phi_{-\lambda}(x),
\phi_\lambda(x)=\phi_\lambda(x^{-1})$ and $\phi_\lambda(e)=1$. Since
$\wp_{-i\rho}(x,n)\equiv 1$ for all $x\in S, n\in N$  and
$\wp_{i\rho}(x,n)=\wp(x,n)$, it follows that
$$\phi_{-i\rho}(x)=\int_N\wp_{i\rho}(e,n)=\int_NP_1(n)dn=1.$$
We have the following asymptotic estimate of $\phi_\lambda$ (see \cite{ADY}).
For $p\in (0, 2]$, let $\gamma_p=2/p-1$.  Then,
\begin{equation}|\phi_{\alpha+i\gamma_p\rho}(x)|\asymp e^{-(2\rho/p')  |x|}, \ \ \alpha\in \R, 0<p< 2.\label{exact-estimate-1}\end{equation}
From this and (\ref{DR-polar}) it follows that $\phi_\lambda\in L^{p', \infty}(S)$ (respectively $\phi_\lambda\in L^{p', 1}(S)$) if and only if $|\Im \lambda|\le \gamma_p\rho$
(respectively $|\Im \lambda|<\gamma_p\rho$) for $1<p<2$ (see \cite{Ray-Sarkar} for more details).
The estimate above degenerates when $p=2$, i.e. when $\gamma_p=0$ and in this case we have  $\phi_0(x)\asymp (1+|x|) e^{-\rho |x|}$.
If $\lambda\in \R^\times$ and $t\ge 1$ then the Harish-Chandra series for $\phi_\lambda$ implies,
\begin{equation} \phi_\lambda(a_t)=e^{-\rho t}[\hc(\lambda)e^{i\lambda t}+c(-\lambda)e^{-i\lambda t}+E(\lambda,
t)], \text{ where } |E(\lambda, t)|\le C_{\lambda} e^{-2t}.
 \label{phi-delta-asym}
\end{equation}
See \cite[(3.11)]{Ion-Pois-1}) for  a proof of the above for the symmetric spaces. The proof works
{\em mutatis mutandis} for general Damek-Ricci spaces.
From this estimate it follows that $\phi_\lambda\in L^{2, \infty}(S)$ for any $\lambda\in \R^\times$.

We define the spherical Fourier transform $\what{f}$ of a suitable
radial function $f$ as
\begin{equation*}\what{f}(\lambda)=\int_Sf(x)\phi_\lambda(x)dx,\end{equation*} whenever the integral converges.
For $1\le p\le 2$, the $L^p$-Schwartz space $C^p(S)$ is defined (see
\cite{ADY, di-B}) as the set of $C^\infty$-functions on $S$ such that
$$\gamma_{r,D}(f)=\sup_{x\in S}|Df(x)|
\phi_0^{-2/p}(1+|x|)^r<\infty,$$ for all nonnegative integers $r$
and left invariant differential operators $D$ on $S$. Let
$C^p(S)^\#$ be the set of radial functions in $C^p(S)$.    We define the strip
$S_p=\{z\in \C\mid |\Im z|\leq \gamma_p\rho\}$ where $\gamma_p=2/p-1$.  Let
$S_p^\circ$ and $\partial S_p$ respectively be the interior and the
boundary of the strip and  $C^p(\what{S})^\#$ be  the set of even holomorphic functions on
$S_p^\circ$ which are continuous on $\partial S_p$ and satisfy for all nonnegative integers $l, m$,
$$\nu_{l,m}(f)=\sup_{\lambda\in S_p}|\frac{d^l}{d\lambda^l}f(x)|
(1+|\lambda|)^m <\infty.$$ When
$p=2$ then the strip degenerates to the line $\R$ and $C^2(\what{S})^\#$ is defined as  the set of even
Schwartz class functions on $\R$. We topologize $C^p(S)$ and $C^p(\what{S})^\#$ by the seminorms $\gamma_{r, D}$ and by $\nu_{l,m}$ respectively. It is known that  (see
\cite{ADY, di-B})   $f\mapsto \what{f}$ is a topological
isomorphism from $C^p(S)^\#$ to $C^p(\what{S})^\#$ for $1\le p \le 2$.

Apart from the weak $L^p$-norm, we  need some size estimates which are close to $L^p$. Let $B(0, R)=\{x\in S\mid |x|<R\}$ be the geodesic ball of radius $R$. For a function $u$ on $S$ and  $1<p<\infty, 1\le q<\infty$ we  define,
\begin{eqnarray}
M_p(u)&=&\left(\limsup_{R\to\infty}\frac{1}{R}\int_{B(0,R)}|u(x)|^pdx\right)^{1/p},\\
\mathcal A_{p,q}(u)&=&\|\mathcal A_q(u)\|_{p, \infty}, \text{ where } \mathcal A_q(u)(x)=\left(\int_{\partial B(\mathfrak s)}|u(r\omega)|^qd\omega\right)^{1/q}\label{npq-DR}.
\end{eqnarray}

\subsection{Symmetric spaces}

We recall that a rank one Riemannian symmetric space of noncompact type  $X$ can be realized as the quotient space $G/K$ where $G$ is a  connected noncompact semisimple Lie group with finite center and of real rank one and $K$ is a maximal compact subgroup of $G$. We consider the Iwasawa decomposition $G=NAK$. Then   the subgroup $N$  is a $H$-type group. Therefore a rank one Riemannian symmetric space $X=G/K$ can be identified as $NA$ through this  decomposition and  the space $X$ accommodates itself inside the DR spaces as an ``Iwasawa $NA$ group''. The $G$-invariant measure $dx$ on $X$ coincides with the left invariant Haar measure  on $X$ viewed as a $NA$ group. The canonical Riemannian  structure on $X$ as $NA$ group also coincides with the Riemannian structure induced by the Killing form, precisely $\langle Y_1, Y_2\rangle=-B(Y_1, \theta Y_2)$ where $B$ and $\theta$ respectively are the Killing form and Cartan involution of $\mathfrak g$, the Lie algebra of $G$ and $Y_1, Y_2\in \mathfrak g$.
 A function on $X$ can be identified with a function on $G$ which is invariant under right $K$-action.  Through this identification semisimple machineries can be brought forward to $X$, which we shall mention below.

 The group $G$ (and in particular its subgroup $K$) acts naturally on $X$ by left translations. Let $M$ be the centralizer of $A$ in $K$.  Apart from the Iwasawa decomposition $G=NAK$ mentioned above, we shall  use the Iwasawa decomposition $G=KAN$. Through the action of $A$ on $N$ mentioned in the previous subsection (in other words  since  $A$ normalizes $N$) $G$ admits  decompositions $G=KNA$ and $G=ANK$.  It also has the  polar decomposition
$G=K\overline{A^+}K$, where $\overline{A^+}$ is identified with nonnegative real numbers. Using the Iwasawa decomposition $G=KAN$,  we write an
element $x\in G$ uniquely as $k(x)\exp H(x)n(x)$ where $k(x)\in K, n(x)\in N$ and $H(x)\in \mathfrak a$, where $\mathfrak a$ is the Lie algebra of $A$. Let $dg$,  $dk$ and
$dm$ be the Haar measures of $G$,  $K$ and $M$ respectively
with $\int_K\,dk=1$ and $\int_M\,dm=1$ and $dn$ be  as given in subsection 2.2. We have the following
integral formulae corresponding to the two Iwasawa decompositions $G=KAN$, $G=NAK$ and the polar decomposition, which hold for any integrable function:
\begin{equation}
\int_Gf(g)dg=C_1\int_K\int_\R\int_N f(ka_tn)e^{2\rho t}\,dn\,dt\,dk, \int_Gf(g)dg=C_2\int_K\int_\R\int_{N}
f(na_tk)e^{-2\rho t}\,dn\,dt\,dk,
\label{Iwasawa-1}
\end{equation}
and
\begin{equation}
\int_Gf(g)dg=C_3\int_K\int_0^\infty\int_K f(k_1a_tk_2) (\sinh
t)^{m_\gamma}(\sinh 2t)^{m_{2\gamma}}\,dk_1\,dt\,dk_2, \label{polar}
\end{equation}
where $m_\gamma$ and $m_{2\gamma}$ are the dimensions of the root spaces  $\mathfrak g_\gamma$ and $\mathfrak g_{2\gamma}$ respectively,
$\gamma$ being the unique positive indivisible root and $\rho=\frac 12 (m_\gamma+2m_{2\gamma})\gamma$, where $\gamma$ is treated as a positive number by $\gamma(1)=1$.
The constants  $C_1, C_2, C_3$ depend on the normalization of the Haar measures involved.
The formulae above are indeed coincides with (\ref{measure-NA}) and (\ref{DR-polar}) (with $\rho=Q$) when $G/K$ is treated as a $NA$ group. The apparent mismatch
is due to the fact that in  the former  we take $\gamma(1)=1/2$ instead of $\gamma(1)=1$,   to make our formulae consistent with the literature.
As in the previous subsection for a function on $X$, we shall use the notation $J(t)$ to rewrite (\ref{polar}) as $\int_X f(x) dx=\int_K\int_0^\infty f(k_1a_t) J(t)dt\, dk$.

We also note that using well  known estimate  $\sinh t\asymp t
e^t/(1+t), t\ge 0$ it follows from  (\ref{polar}) that
\begin{eqnarray}
\int_G|f(g)|dg&\asymp&C_3\int_K\int_0^1\int_K |f(k_1a_tk_2)|
t^{d-1}\,dk_1\,dt \,dk_2\nonumber \\&+&C_4 \int_K\int_1^\infty\int_K
|f(k_1a_tk_2)|e^{2\rho t}\,dk_1\,dt\,dk_2 \label{polar-2}
\end{eqnarray} where $d=m_\alpha+m_{2\alpha}+1$.

It is well known that the {\em maximal distinguished boundary} (or boundary for short) of the symmetric space  $X=G/K$ has two different, albeit essentially equivalent, realizations which we obtain  through the Iwasawa decomposition of  $G=KAN$. The compact boundary is $K/M$ and the noncompact one is the nilpotent group $N$. There is a natural correspondence between these two boundaries if we leave out  an appropriate  set of measure zero.
If  $G/K$ is realized as an Iwasawa $NA$ group, then as done in the previous subsection,  we consider $N$ as the boundary and  deal with the Poisson transform $\wp_\lambda$.  We shall define below  the Poisson transform $\mathcal P_\lambda$ considering the compact boundary $K/M$. We refer to \cite[pp. 418--419]{ACB}, for  relation between $\mathcal P_\lambda$   and $\wp_\lambda$ defined in the previous subsection.

For $\lambda\in \C$, the complex power of the
Poisson kernel: $x\mapsto  e^{-(i\lambda+\rho) H(x^{-1})}$ is an
eigenfunction of the Laplace Beltrami operator $\Delta$ with
eigenvalue $-(\lambda^2+\rho^2)$. For any $\lambda\in \C$ and $F\in
L^1(K/M)$ we define the  Poisson transform $\mathcal P_\lambda$ of $F$
by (see \cite[p. 279]{Helga-2}) by
$$\mathcal P_\lambda  F(x)=\int_{K/M} F(k)e^{-(i\lambda+\rho) H(x^{-1}k)} dk \text{ for  } x\in X.$$ Then,
$$\Delta \mathcal P_\lambda F=-(\lambda^2+\rho^2)\mathcal P_\lambda F.$$

We recall that for these Iwasawa $NA$ groups a  function is radial if and only if $f(kx)=f(x)$ for all $k\in K$ and $x\in X$. The radialization operator $R$ takes the simpler form: $Rf(x)=\int_Kf(kx)dk$.

For any $\lambda\in \C$  the elementary spherical function
$\phi_\lambda$ defined in subsection 2.2 has the following alternative expression,
$$\phi_\lambda(x)=\mathcal P_\lambda 1(x)=\int_{K/M} e^{-(i\lambda+\rho)H(xk)}\,dk \text{ for all } x\in G,$$ where by $1$ we denote the constant function $1$ on $K/M$ (see \cite{ACB, KRS}).
It is clear that on $X$ the function (defined in subsection 2.2) $\mathcal A_q(u)(x)=\left(\int_{K/M}|u(kx)|^qdk\right)^{1/q}$.

\section{Sharpness of  Theorems \ref{thm-X-2}, \ref{thm-X-p}}
In this section we shall try to motivate the formulation of the theorems
stated in the introduction and establish their sharpness by answering the following natural questions:

(a) Does Theorem A hold true when $\alpha=0$? Does Theorem B hold true when $\alpha=\beta\pm i\gamma_p\rho$ for $\beta\neq 0$?

(b) In Theorem A (respectively in Theorem B),  is it possible to substitute $L^{2,\infty}$-norm (respectively $L^{p', \infty}$-norm) by any other Lorentz norms?

(c) Is it necessary to use both positive and negative integral powers of $\Delta$ in Theorem A ?
\subsection{Estimates of the Poisson transform}
We need to start with the basic $L^p$-behaviour of the Poisson transform as  the  Poisson transforms (of functions or functionals) form the  set of eigenfunctions of the
Laplacian.  From the Kunze-Stein
phenomenon and the Herz's principe de majoration (see \cite{Lo-Ry,
Cow-Herz}) it follows that for $1\le p<2$, $p\le q\le p'$ and $\alpha\in \R$,
the Poisson transform  on $X$ satisfies the following estimates:
\begin{equation}\|\mathcal P_{\alpha+i\gamma_q\rho}F\|_{p', \infty}\le C \|F\|_{L^{q'}(K/M)}.
\label{poisson-estimate-X}
\end{equation}

Recently similar estimates for the  Damek-Ricci spaces $S$ were also obtained by the authors of this paper (see \cite{KRS, Ray-Sarkar}): For $1\le p<2$, $p\le q\le p'$ and $\alpha\in \R$,
\begin{equation}\|{\mathfrak P}_{\alpha+i\gamma_q\rho}F\|_{p', \infty}\le C \|F\|_{L^{q'}(N)}.
\label{poisson-estimate-S}
\end{equation}

From the estimate of $\phi_0$ given in section 2, it follows that
$\phi_0\not\in L^{2, \infty}(S)$ which obviates an
 analogue of (\ref{poisson-estimate-X}) and (\ref{poisson-estimate-S}) for $p=2$.  However, from (\ref{phi-delta-asym}), it follows that for real $\lambda\neq 0$,
 $\phi_{\lambda}\in L^{2, \infty}(S)$. Therefore one would expect an inequality  on $X$ and $S$ respectively of the form:
$$\|\mathcal P_{\lambda}F\|_{2, \infty}\le C(\lambda) \|F\|_{L^{2}(K/M)},
\|{\mathfrak P}_{\lambda}F\|_{2, \infty}\le C(\lambda)
\|F\|_{L^{2}(N)}\text{ for } \lambda\in \R^\times.$$ At this point
of time such an inequality is not known.
It is
well known that there does not exists any eigenfunction of $\Delta$
which is in $L^p(S)$ for $p\le 2$.
But we have observed above that there are $L^{2, \infty}$-eigenfunctions of $\Delta$
(e.g. $\phi_\lambda$, $\lambda\in \R^\times$).  We also have the following result.

\begin{proposition} \label{negative-results}
Let $u$ be a nonzero function on $X$.

{\rm(i)} If $\Delta u = -\rho^2 u$ then $u\not\in L^{2, \infty}(X)$.

{\rm(ii)} If $\Delta u = -(\lambda^2+\rho^2)u$ for some $\lambda\in \R^\times$, then $u\not\in L^{2,q}(X)$
for $q<\infty$.

{\rm(iii)} If for some $1<p<2$, $\Delta u = -[(\beta\pm i\gamma_p\rho)^2+\rho^2]u$ for  $\beta\in \R$ then $u\not\in L^{q', r}(X)$ if one of these two conditions is satisfied: {\em (a)} $q>p$,  {\em (b)} if $q=p$ and $r<\infty$.
\end{proposition}
\begin{proof}
We suppose that $\Delta u= -(\lambda^2+\rho^2) u$ for some $\lambda\in \C$ and $u(x_0)\neq 0$ for some $x_0\in X$.
Then $f(y)=\int_Ku(x_0ky)dk$ satisfies $f(y)=\phi_\lambda(y)u(x_0)$ (see \cite[p. 402]{Helga-3}) and hence $f$ is a  radial eigenfunction of $\Delta$ with the eigenvalue $-(\lambda^2+\rho^2)$.
Thus to prove  (i) and (iii) respectively, it is  enough to show that $\phi_0\not\in L^{2, \infty}(X)$ and
 $\phi_{\beta\pm i\gamma_p\rho}\not\in L^{q',r}(X)$ for $q,r$ as in (iii), which are  clear from the estimates of $\phi_0$ and $\phi_{\beta-i\gamma_p\rho}$ given in section 2. Similarly for (ii) it is enough to show that $\phi_\lambda$ with $\lambda\in \R^\times$ is not in $L^{2,q}(X)$ for  any $q<\infty$, which we shall prove below.

We note that  $\phi_\lambda\not\in L^{2, q}(X)$ is equivalent to
$\phi_\lambda\not\in L^{2, q}(\R^+, J(t)dt)$ where $J(t)$ is
the Jacobian in the polar decomposition (see (\ref{polar-2})). Since
$\phi_\lambda$ is a continuous function, and $J(t)\asymp e^{2\rho t}$
when $t\ge 1$, it suffices to show that for $t\ge 1$, $t\mapsto
\phi_\lambda(a_t)$ is not in $L^{2, q}((1, \infty), e^{2\rho t}dt)$.
 Since,
$E(\lambda,  t)\le C_\lambda e^{-2t}$ for $t\ge 1$ in
(see (\ref{phi-delta-asym})), the assumption that $\phi_\lambda(a_t)$ is
in $L^{2, q}((1, \infty), e^{2\rho t}dt )$ will imply that
$g(t)=e^{-\rho
t}(\hc(\lambda)e^{it\lambda}+\hc(-\lambda)e^{-it\lambda})=e^{-\rho
t}2\Re(\hc(\lambda)e^{it\lambda})$ is in $L^{2, q}((1, \infty),
e^{2\rho t}dt)$.
Let $\hc(\lambda)=a(\lambda)+ib(\lambda)$. Then $g(t)=2e^{-\rho
t}(a(\lambda)\cos \lambda t-b(\lambda)\sin \lambda t)$. Since
the translation operator  $f(t)\mapsto f(t+\pi/2\lambda)$ is $(p,p)$
for any $p\ge 1$ in the measure space $((1, \infty), e^{2\rho
t}dt)$, we get  by interpolation (see \cite[p. 197]{S-W}) that the
translation operator  is bounded on $L^{2, q}(X)$. Thus we get
$g(\bullet+\pi/2\lambda)\in L^{2,q}((1, \infty), e^{2\rho t}dt)$ and
hence $b(\lambda)g(t)+a(\lambda)g(t+\pi/2\lambda)\in L^{2,q}((1,
\infty), e^{2\rho t}dt)$. Since $g(t+\pi/2\lambda)=-2e^{-\rho
t}(a(\lambda)\sin\lambda t+ b(\lambda)\cos \lambda t)$, it follows that
$b(\lambda)g(t)+a(\lambda)g(t+\pi/2\lambda)=-2e^{-\rho t}
(b(\lambda)^2+a(\lambda)^2)\sin \lambda t \in L^{2,q}((1, \infty),
e^{2\rho t}dt)$, i.e. $e^{-\rho t}\sin \lambda t\in L^{2,q}((1, \infty),
e^{2\rho t}dt)$. Similarly we can show that $e^{-\rho t}\cos
\lambda t\in L^{2, q}((1, \infty), e^{2\rho t}dt)$. These two together imply that
$e^{-\rho t} \in L^{2,q}((1, \infty), e^{2\rho t}dt)$, which is false
as can be verified by direct computation.
\end{proof}

\begin{remark}
From (i) of the proposition above it is straightforward to see that the hypothesis of
Theorem A with $\alpha=0$ cannot yield any nonzero eigenfunction of $\Delta$. Similarly from (ii) of the proposition it follows that Theorem A with $L^{2,q}$-norm, $q<\infty$ replacing the $L^{2,\infty}$-norm will not have any nonzero solution either. In the same way
(iii) discards the use of $L^{p',r}$-norm with $r<\infty$ and that of $L^{q',r}$-norm with $q>p$ in Theorem B.
It is also noted above that $p<2$ cannot be used in Theorem A and Theorem B.
These along with the counter examples constructed in the next subsection  will complete answering the questions.
\end{remark}

\subsection{Counter examples}
Let  $1<p<2.$ We will show that   if we substitute
$i\gamma_{p'}\rho$ by $\beta\pm i\gamma_{p'}\rho$  in Theorem B with
$\beta\in\R^\times$, then there exists a measurable function $f$
satisfying the hypothesis of the theorem but $f$ is not an
eigenfunction of $\Delta.$ We shall consider only $\alpha=\beta+i\gamma_p\rho$.
The case $\alpha=\beta-i\gamma_p\rho$ will be analogous. To show this, we will appeal to the
description of the $L^{p_1}$-spectrum of $\Delta$ for $1\leq
p_1<\infty.$ It is known that (see \cite{Lo-Ry-2,  Tay, ADY})
the $L^{p_1}$-spectrum $\sigma_{p_1}$ of $\Delta$ is the image of
the set $S_{p_1}$ under the map $\Lambda(z)=-(z^2+\rho^2).$
Precisely, $\sigma_{p_1}$ is given by the parabolic region
\begin{equation}\sigma_{p_1}=\{-(z^2+\rho^2)\mid |\Im z|\leq
|\gamma_{p_1}\rho|\}=\sigma_{p_1'}.\label{spectrum}\end{equation}
 We note that for $p$ as above if $p<q<2$
then $\gamma_q<\gamma_p$, hence  $\Lambda
(\beta+i\gamma_q\rho)\in \sigma_p$  and
$\phi_{\beta+i\gamma_q\rho}\in L^{p', 1}(X)\subset L^{p',\infty}(X)$ (see section 2).
Since $\beta\neq 0$ we can choose $r$ such that $p<q<r<2$ and real numbers $s,t$
 such that $|\Lambda (s+i\gamma_{q'}\rho )|=|\Lambda
(t+i\gamma_{r'}\rho )|=|\Lambda (\beta+i\gamma_{p'}\rho )|$.
Figure 1 below   explains the situation.
\begin{figure}[!ht]
$$
\begin{minipage}[b]{0.5\linewidth}
\centering
\includegraphics[width=3.2in, height=2.2in,
angle=0]{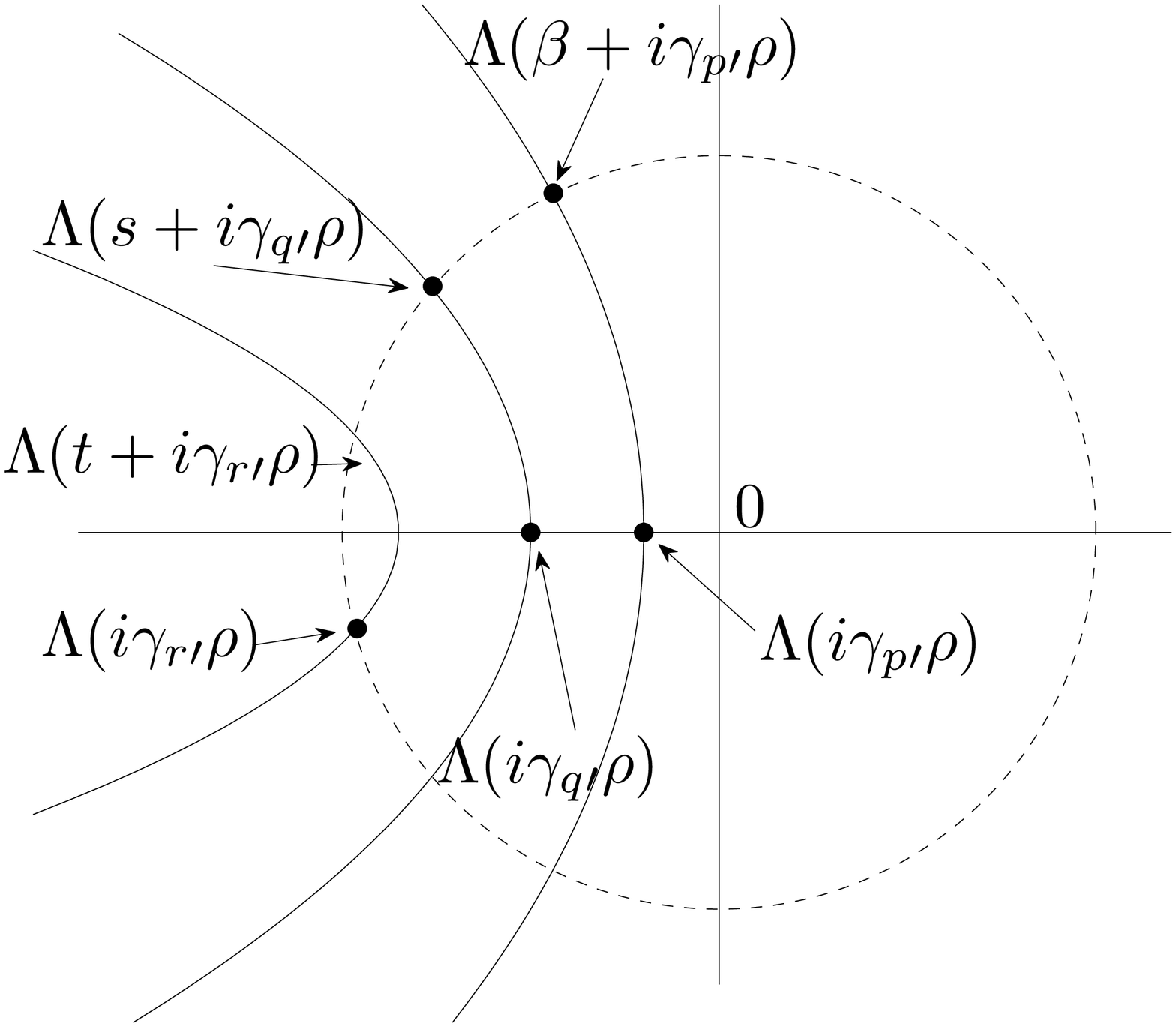}
\caption{$  1 < p < q< 2$}
\label{fig:fig1}
\end{minipage}
\hspace{0.1cm}
\begin{minipage}[b]{0.5\linewidth}
\centering
\includegraphics[width=3.2in, height=2.2in, angle=0]{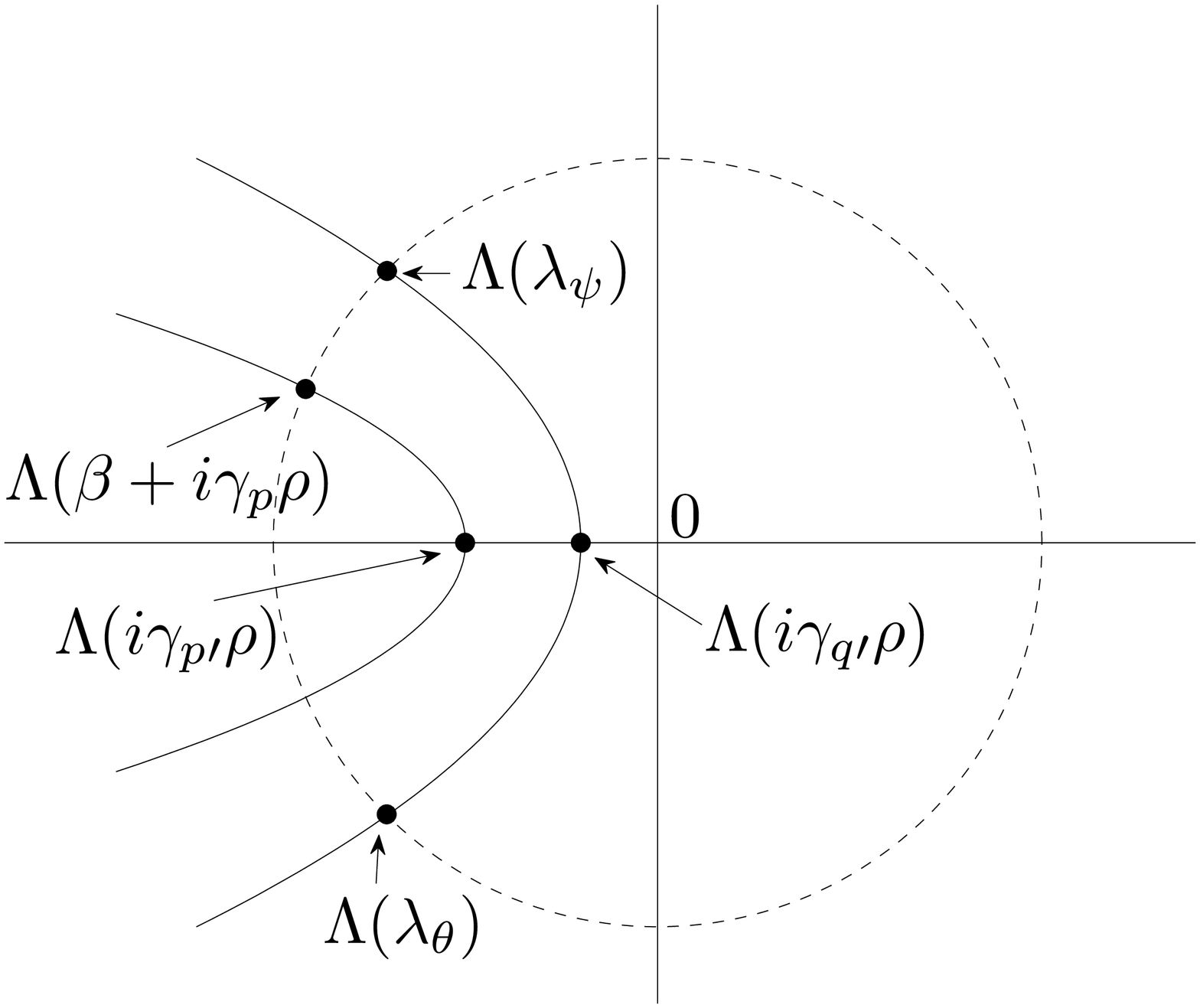}
\caption{$1<q<p<2$}
\label{fig:fig2}
\end{minipage}
$$
\end{figure}

Hence there exists $\theta$ and $\psi$ in $(0, 2\pi )$ such that
\begin{equation}\Lambda (s+i\gamma_{q'}\rho )e^{-i\theta}=\Lambda
(t+i\gamma_{r'}\rho )e^{-i\psi}=\Lambda (\beta+i\gamma_{p'}\rho
).\label{circle}\end{equation} We
define $$f(x)=\phi_{t+i\gamma_{r'}\rho}(x)+\phi_{s+i\gamma_{q'}\rho}(x).$$  Using (\ref{circle}) it
follows that
\begin{eqnarray*}\Delta^k f(x)&=&(\Lambda(t+i\gamma_{r'}\rho))^k\phi_{t+i\gamma_{r'}
\rho}+(\Lambda(s+i\gamma_{q'}\rho))^k\phi_{s+i\gamma_{q'}\rho}\\
&=&e^{ik\psi}(\Lambda(\beta+i\gamma_{p'}\rho))^k\phi_{t+i\gamma_{r'}\rho}
+e^{ik\theta}(\Lambda(\beta+i\gamma_{p'}\rho)^k\phi_{s+i\gamma_{q'}\rho}\\
&=&(\Lambda(\beta+i\gamma_{p'}\rho))^k f(x).
\end{eqnarray*}
Therefore $|\Delta^kf|\leq |\Lambda(\beta+i\gamma_p\rho)|^k(|\phi_{i\gamma_r\rho}|+|\phi_{i\gamma_q\rho}|$. Since both $\phi_{i\gamma_r\rho}$ and $\phi_{i\gamma_q\rho}$ are in $L^{p', \infty}(X)$ (see section 2), $f$ satisfies the hypothesis of Theorem B.
 It is clear
that $f$ is not an eigenfunction of $\Delta.$

Next we shall show that in Theorem B, if we take $\alpha=\beta\pm i\gamma_p\rho$, $\beta\in \R$ and substitute
$L^{p', \infty}$-norm by $L^{q', r}$-norm with $1\le q<p<2$, then there are functions $f$ which satisfy the hypothesis, but they  are not eigenfunctions of $\Delta$. As above we shall only consider the case $\alpha=\beta+i\gamma_p\rho$. Indeed (using the notation $\Lambda$ given above) there exists $\theta,\psi$ and $\lambda_\theta, \lambda_\psi\in S_q^\circ$ such that $\Lambda(\lambda_\psi)e^{-i\psi}=\Lambda(\lambda_\theta)e^{-i\theta}=\Lambda(\alpha)=-(\alpha^2+\rho^2)$ (see Figure 2). As above we define $f=\phi_{\lambda_\theta}+\phi_{\lambda_\psi}$.  Then
\[\Delta^k f=(\Lambda(\lambda_\theta))^k\phi_{\lambda_\theta}+(\Lambda(\lambda_\psi))^k\phi_{\lambda_\psi}=
\Lambda(\alpha)^k[e^{ik\theta}\phi_{\lambda_\theta}+e^{ik\psi}\phi_{\lambda_\psi}].\] From this it is clear that
$f$ satisfies the hypothesis of Theorem B and $f$ is not an eigenfunction of $\Delta$.
A similar construction will show that in Theorem A,
if we substitute $L^{2, \infty}$-norm  by   $L^{p', r}$-norm with $2<p'<\infty, 1\le r\le
\infty$ or by $L^\infty$-norm, then there exist functions which satisfy the hypothesis, despite  not being eigenfunctions of $\Delta$. This completes answering the questions  and thereby establishes the sharpness of Theorem
\ref{thm-X-2} and Theorem \ref{thm-X-p}.

Lastly we shall show that unlike Theorem \ref{thm-X-p}, it is necessary to use positive as well as negative integral powers of $\Delta$ in Theorem \ref{thm-X-2}.
Let $f=\phi_{\lambda_1}+\phi_{\lambda_2}$ for $\lambda_1, \lambda_2\in \R^\times, \lambda_1\neq \lambda_2$
and  $(\lambda_i^2+\rho^2)/(\alpha^2+\rho^2)<1$ for
$i=1,2$. Then
$\Delta^kf=-(\lambda_1^2+\rho^2)^k\phi_{\lambda_1}-(\lambda_2^2+\rho^2)^k\phi_{\lambda_2}$ and hence
$|\Delta^kf|\le (|\phi_{\lambda_1}|+|\phi_{\lambda_2}|) (\alpha^2+\rho^2)^k$. Therefore
$\|\Delta^k f\|_{2, \infty} \le M(\alpha^2+\rho^2)^k$ where
 $M=\|\phi_{\lambda_1}\|_{2, \infty}+\|\phi_{\lambda_2}\|_{2,
\infty}<\infty$. But  it is
clear that $f$ is not an eigenfunction of $\Delta$.

\section{Characterization of eigenfunction}
In this section we shall focus mainly on Iwasawa $NA$ groups,  in
other words on the Riemannian symmetric spaces $X$ of noncompact
type with real rank one. The purpose  is to prove certain
representation theorems for eigendistribution of the Laplace-
Beltrami operator $\Delta$ on $X=G/K.$ In particular we shall
generalize Theorem \ref{iobs} (b)  for $p\in (1,2)$ which was
conjectured in \cite{Bou-Sami} (see Theorem \ref{samib} below). To
put things in proper perspective we shall first look at the
available results. However being  comprehensible on this vast topic is beyond the scope of this article.
Instead, we shall restrict our attention only on those results which are related to the main theme of this paper. Throughout this section for a complex number $\lambda$,
\[E_{\lambda}=\{u\in C^{\infty}(X)\mid \Delta u=-(\lambda^2+\rho^2)u\}.\]

\subsection{A brief Survey}
We have the following theorem on $X$ which can be viewed as an
analogue of a result (\cite[p. 50, Theorem 2.5]{S-W}) on the upper half
plane $\R_{n+1}^+=\{(x,y)\in\R^{n+1}\mid x\in\R^n,y>0\}$.

\begin{theorem}[Furstenberg \cite{F}, Knapp-Williamson \cite{KW}] A harmonic function
$u: X\rightarrow \C$ satisfies
$$\sup_{t>0} \left(\int_{K/M} |u(ka_t)|^pdk\right)^{1/p}<\infty, 1<p<\infty, \ \ \ (\text{respectively, } \|u\|_\infty<\infty)$$ if
and only if there exists $f\in L^p(K/M)$ $($respectively $f\in
L^\infty(K/M)$$)$ such that  $u=\mathcal P_{-i\rho} f$. \label{Fur-KW}
\end{theorem}
See also \cite{Stoll}. An analogue of the case $p=\infty$ of the
result above  for general $NA$ groups  was proved by Damek in
\cite{Dam1}. It is known that the theorem above is not true for bounded eigenfunctions of $\Delta$ with nonzero eigenvalues (see \cite{Gre}).

Going towards the representation of eigenfunctions of $\Delta$ on
$X$ with nonzero eigenvalue, we recall that  an eigendistribution $u$ of $\Delta$ on $X$ is
a real analytic function and is Poisson transform of an {\em
analytic functional} $T$ defined on $K/M$ (see \cite{H1}, \cite{H2},
\cite{KKM}). Lewis \cite{L}  puts additional condition on $u$ to
ensure that the analytic functional $T$ is a distribution on $K/M$.
\begin{theorem}[Lewis] If $u$ is an eigenfunction of $\Delta$ then $u$ is Poisson
transform of a distribution $T$ defined on $K/M$ if and only if
there exists $\beta>0$ such that $|u(ka_t)|\leq C e^{\beta t}$ for
all $k\in K/M$ and $t>0$.\label{lewis}
\end{theorem}

We also have (\cite[Theorem 3.2 (v)]{BOS}),
\begin{lemma} Let $u\in E_\lambda(X)$ with $\Im \lambda<0$ or $\lambda=0$.
If $u$ is a Poisson transform of a distribution $T$ defined on $K/M$, then $u(ka_t)/\phi_\lambda(a_t)$ converges to $T$ in the sense of distribution as $t\to \infty$.
\label{lemma-Ben}
\end{lemma}
Sj\"{o}gren (\cite[Theorem 6.1]{Sj}) determines the size estimates
on $u$ which are sufficient to imply that the distribution $T$ is
actually given by a function.
\begin{theorem}[Sj\"{o}gren, \cite{Sj}] Let $u\in E_{\lambda}$ for some $\lambda\in \C$ with $\Im\lambda<0$ or $\lambda=0.$
For $1<p\leq\infty,$ and  $\beta>0$, the function $ka_t\mapsto
\phi_{\Im\lambda}(a_t)^{-1}e^{-\beta t/p}u(ka_t)$ belongs to
$L^{p,\infty}(X, m_{\beta})$ if and only if $u=\mathcal P_{\lambda}f$ for
some $f\in L^p(K/M).$ Here $dm_{\beta}(x)=dm_{\beta}(ka_t)=e^{(\beta - 2\rho)
t}J(t)dk dt,$ $t>0,$ $k\in K/M$.  \label{sj}
\end{theorem}
Taking $\beta=2\rho$ and $p=q',$ $1\leq q<2$ in Theorem \ref{sj} we
get an $L^p$-version of the result of Furstenberg mentioned above.
\begin{corollary}
Let $1\leq q<2$ and $\lambda=\alpha-i\gamma_q\rho,$ $\alpha\in\R$
and $u\in E_{\lambda}.$ Then $u\in L^{q',\infty}(X)$ if and only if
$u=\mathcal P_{\lambda}f$ for some $f\in L^{q'}(K/M).$ In particular, if
$q=1$ then bounded eigenfunctions of $\Delta$ with eigenvalue
$\alpha(2i\rho-\alpha)$ are Poisson transform of bounded functions
on $K/M.$ \label{lr}
\end{corollary}

The case $1<q<2, \alpha=0$ of the result above was proved
independently in \cite{Lo-Ry}.
The following result in \cite{BOS}  generalizes  Theorem
\ref{Fur-KW} for eigenfunctions of $\Delta$ with nonzero
eigenvalues.
\begin{theorem}[Ben Sa\"{i}d et. al.]
Let $\Im\lambda<0$ or $\lambda=0,$ $1< p\leq\infty$ and $u\in
E_{\lambda}.$ Then $u=\mathcal P_{\lambda}f$ for some $f\in L^p(K/M)$ if and
only if
$\sup_{t>0}\phi_{\Im\lambda}(a_t)^{-1}\left(\int_{K/M}|u(ka_t)|^pdk\right)^{1/p}<\infty$ \ $($with usual modification for $L^\infty$-norm$)$.
If $p=1$ then $u=\mathcal P_{\lambda}\mu$ for some signed measure $\mu$ on
$K/M$.
\label{sos}
\end{theorem}
We observe that all these  results leave out the case
$\lambda\in\R^\times$.
 Motivated by a work of Strichartz
(\cite{St1}) this  was taken up by
 Ionescu  and
Boussejra et al. (\cite{Ion-Pois-1, Bou-Sami}) and their works reveal  that
the oscillatory behaviour of $\phi_\lambda, \lambda\in
\R^\times$ plays a critical role in this case.
\begin{theorem}
 [Ionescu, Boussejra et al.] Suppose that $u\in E_{\lambda}$ with
 $\lambda\in\R^\times$.
 \begin{enumerate}
 \item[(a)]
Then $u=\mathcal P_{\lambda}f$ for some $f\in L^2(K/M)$ if and only if
$M_2(u)<\infty.$ Moreover, in this case
$$M_2(u)=|\hc(\lambda)|\|f\|_{L^2(K/M)} \text{ for all } f\in L^2(K/M).$$
\item[(b)] For  $p\geq 2$ and $X$ a hyperbolic space over $\R,\C$
or $\mathbb H$, $u=\mathcal P_{\lambda}f$ for some $f\in L^p(K/M)$ if and
only if $\sup_{t>0}e^{\rho t}\mathcal A_p(u)(a_t)<\infty.$ \end{enumerate} \label{iobs}
\end{theorem}
It was conjectured in \cite{Bou-Sami} that (b)  holds also for all $p\in (1, 2)$, which we shall prove
in Theorem \ref{samib}.

\subsection{Preparatory Lemmas}
\begin{lemma}Let $u\in E_{\lambda},$ with $\lambda\in\C$ and
$1<p<\infty, 1\le q<\infty$. We denote by $\|\cdot\|$ any of these
three norms $\|\cdot\|_{p, \infty}$, $M_p$ or $\mathcal A_{p,q}$. If
$\|u\|<\infty$ then there exists $\alpha>0$, $C>0$, both of  which
may depend on $u$ such that
\begin{equation} |u(ka_s)|\leq C\|u\|e^{\alpha s} \text{ for all } s>0 \end{equation}
\label{exponential}
\end{lemma}

\begin{proof}
Since $u$ is continuous it is enough to assume
that $s$ is large. We shall first take the case $\|u\|_{p,
\infty}<\infty$.  We consider the ball $B(0,e^{-s})=\{k_1a_r \mid k_1\in
K, 0<r<e^{-s}\},\  s>1$ and using polar coordinates we get
 \begin{equation}
|B(0,e^{-s})|\asymp\int_0^{e^{-s}}r^{d-1}dr=C_d e^{-ds},\label{ballmeasure}
\end{equation} where $d=\dim X$. By generalized mean value property of
eigenfunctions of the Laplacian (\cite{Helga-2}, p.402) we have \begin{equation}
u(g)\phi_{\lambda}(x)=\int_Ku(gkx)dk,\,  g\in G, x\in
X.\label{gmvt} \end{equation}
Since $\phi_{\lambda}(e)=1$ it follows that for large $s$, we have
$|\phi_{\lambda}(x)|\geq C$ for all $x\in B(0,e^{-s})$. Therefore
from (\ref{gmvt})  we conclude that \begin{equation} |u(g)|\leq
C\int_K|u(gkx)|dk, \text{ for all } x\in B(0,e^{-s}) g=k_1a_s.\label{meanineq}
\end{equation}
 Integrating both sides of
(\ref{meanineq}) over $B(0,e^{-s})$ and noting that $B(0,e^{-s})$ is
$K$-invariant, we get
\begin{eqnarray*}
|B(0,e^{-s})||u(g)|&\leq
&C\int_{B(0,e^{-s})}\int_K|u(gkx)|dkdx\\
&=&C\int_{B(0,e^{-s})}|u(gx)|dx\nonumber\\
&\leq
&C\|u\|_{L^{p,\infty}(X)}\|\chi_{B(0,e^{-s})}\|_{L^{p,1}(X)}\\
&=&C\|u\|_{L^{p,\infty}(X)}|B(0,e^{-s})|^{1/p}.\nonumber
\end{eqnarray*}
From above and (\ref{ballmeasure})  we get  for large $s$
$$|u(g)|=|u(k_1a_s)|\leq C\|u\|_{L^{p,\infty}(X)}e^{\frac{d}{p'}s}.$$
This completes the proof  when $\|u\|_{p, \infty}<\infty$.

Now we assume that $M_p(u)<\infty$. As above in this case also it
follows that if $g=k_1a_s$ then \begin{equation}|B(0,e^{-s})|u(g)|\leq
C\int_{B(0,e^{-s})}|u(gx)|dx\leq
C\int_{B(0,e^{-s}+s)}|u(x)|dx,\label{eqn1} \end{equation} as $d(0,x)\leq
e^{-s}$ implies $d(0,gx)\leq d(0,g)+d(g,gx)\leq
s+ e^{-s}$ and consequently $g B(0,e^{-s})\subset
B(0,e^{-s}+s).$ By H\"{o}lder's inequality from (\ref{eqn1}), we get
\begin{equation} |B(0,e^{-s})|u(ka_s)|\leq
\left(\int_{B(0,e^{-s}+s)}|u(x)|^{p}dx\right)^{1/p}|B(0,e^{-s}+s)|^{1/p'}.\label{eqn3}
\end{equation} Since $|B(0,e^{-s})|=e^{-ds}$ and for large $s,$
$|B(0,e^{-s}+s)|\leq e^{4\rho s}$, it follows from (\ref{eqn3}) that
$$|u(ka_s)|\leq ce^{(d+\frac{1}{p})s}e^{\frac{4\rho s}{p'}}M_p(u), \text{ for all large } s.$$
Lastly we assume that $\mathcal A_{p,q}(u)<\infty$.  From
(\ref{eqn1}) it follows that for all large $s$
\begin{equation}|B(0,e^{-s})|u(g)|\leq C\int
_{B(0,e^{-s}+s)}|u(x)|dx\leq C\int_{B(0,2s)}|u(x)|dx\leq C\int_0^{2s}\int_K|u(ka_t)|J(t)dtdk\label{eqn2}\end{equation}
Hence from (\ref{eqn2}) we get,
\begin{eqnarray}|B(0,e^{-s})|u(g)|&\leq&\int_0^{2s}\left(\int_K|u(ka_t)|^q dk\right)^{1/q}J(t)dt\le \mathcal A_{p,q}(u)|B(0,2s)|^{1/p'}.
\label{eqn2-1}
\end{eqnarray} Since for large
$s,$ $|B(0,e^{2s})|\asymp C e^{4\rho s}$ it follows from
(\ref{ballmeasure}) and (\ref{eqn2-1}) that
$$|u(ka_s)|\leq C e^{d s}e^{\frac{4\rho s}{p'}}\mathcal A_{p,q}(u)\leq
C \mathcal A_{p,q}(u)e^{\alpha s} \text{ for all large } s.$$
This completes the proof.
\end{proof}
The next lemma compares $M_p(u)$ and $\|u\|_{p,\infty}$ when $u$ is
an eigenfunctions of $\Delta$. The technique  can be traced back to
\cite{Sj}.
\begin{lemma}
If $u\in C(X)\cap L^{p,\infty}(X),$ $1<p<\infty$ and there exist
$\alpha>0,$ $C>0$ such that for all $t>0,$ and $k\in K,$
$|u(ka_t)|\leq C\|u\|_{L^{p,\infty}(X)}e^{\alpha t}$ then for all $R>0$,
\begin{equation}
\int_{B(0,R)}|u(x)|^pdx\leq
C_{\alpha,p}\|u\|_{L^{p,\infty}(X)}R,\quad\text{equivalently}\quad  M_p(u)\leq
C_{\alpha,p}\|u\|_{L^{p,\infty}(X)}. \label{basic} \end{equation}
\label{mp}
\end{lemma}

\begin{proof}
Let $u^*$ be the  decreasing rearrangement of $u$. Then   it follows
from the definition  that \begin{equation} u^{*}(s)^p\leq
\frac{1}{s}\|u\|_{p,\infty}^p, \text{ for all }
s>0. \label{weaklp}\end{equation} We also have (\cite[p. 64, Proposition
1.4.5, (7), (11)]{Graf})
\begin{equation} \int_{B(0,R)}|u(x)|^pdx\leq
c\int_0^\infty (\chi_{B(0,R)}|u|)^{*}(t)^p dt.\label{rearrange} \end{equation}
Since $u$ grows at most exponentially, we have,
\begin{equation}|\chi_{B(0,R)}(x)u(x)|\leq \|u\|_{L^{p,\infty}(X)}e^{\alpha
R}\chi_{B(0,R)}(x).\label{rearrange-1}\end{equation} As $(\chi_{B(0,R)})^{*}=
\chi_{(0, |B(0,R)|)}$ it follows from (\ref{rearrange-1}) that
(\cite{Graf}, Proposition 1.4.5, (3)) \begin{equation}
(\chi_{B(0,R)}|u|)^{*}(t)^p\leq C\|u\|_{p,\infty}^pe^{\alpha p
R}\chi_{(0, |B(0,R)|)}(t).\label{rearrange-est} \end{equation}

When  $R>1$, we have from (\ref{weaklp}), (\ref{rearrange}) and
(\ref{rearrange-est})
\begin{eqnarray} \int_{B(0,R)}|u(x)|^pdx&\leq&
C\int_0^{|B(0,R)|}\min\left\{\|u\|_{p,\infty}^pe^{p\alpha
R},\|u\|_{p,\infty}^p\frac{1}{t}\right\}dt\nonumber\\
&\leq& C\|u\|_{p,\infty}^p\int_0^{e^{2\rho r}}\min\left\{e^{p\alpha
R},\frac{1}{t}\right\}dt\nonumber\\
&=&C\|u\|_{p,\infty}^p\left(\int_0^{e^{-p\alpha R}}e^{p\alpha
R}dt+\int_{e^{-p\alpha R}}^{e^{2\rho
r}}\frac{dt}{t}\right)\nonumber\\
&=&C\|u\|_{p,\infty}^p\left(1+2\rho R+p\alpha
R\right)\nonumber\\
&\leq & C_p\|u\|_{p,\infty}^p R.\label{Rbigger1} \end{eqnarray}

When  $R\leq 1$,   $|B(0,R)|\asymp R^{d-1}$  and  for
all $t\in [0,R^{d-1}]$ we have,
$$\min\{e^{p\alpha R},\frac{1}{t}\}\leq Ce^{p\alpha R}\leq
Ce^{p\alpha}.$$ Hence \begin{equation} \int_0^{|B(0,R)|}\min\{e^{p\alpha
R},\frac{1}{t}\}dt\leq C_{\alpha,p}\int_0^{R^{d-1}}e^{p\alpha}dt\leq
C_{\alpha,p}R,\label{Rlesser1} \end{equation} as $d\geq 2.$ Combining
(\ref{Rbigger1}) and (\ref{Rlesser1}) it follows that
$$\int_{B(0,R)}|u(x)|^pdx\leq C_{\alpha,p}\|u\|_{p,\infty}^pR, \text{ for all  } R>0.\qquad\qquad\qquad\qquad\qquad\qquad\qquad\qquad\qquad\qquad\qedhere$$
\end{proof}
\subsection{Main results on characterization}
We need the following results.

\noindent (I)  Lt $\lambda=\alpha+i\gamma_{p'}\rho$, $1<p<2$, $\alpha\in \R$. For $t>0$ and a measurable function $f$ on $K/M$ we define the maximal function
\[\wtilde{f}(k)=\sup_{t>0}\phi_{i\gamma_p\rho}(a_t)^{-1} (e^{-(i\lambda+\rho)H(a_t^{-1}\cdot)}\ast f)(k)\] where the convolution is on $K$.
From this we have the pointwise estimate
\begin{align}\mathcal P_{\alpha+i\gamma_{p'}\rho}f(ka_t)|&\le C\phi_{i\gamma_p\rho}(a_t)\wtilde{f}(k).\label{mike}\end{align}
Michelson (\cite{M}) proved that $\wtilde{f}$ satisfies,
$\|\wtilde{f}\|_{L^r(K/M)}\le C\|f\|_{L^r(K/M)}, 1<r<\infty$. (see also \cite{Lo-Ry}).

\noindent (II)  For a
nonnegative continuous function $\Phi$ defined on $[1,\infty )$ if
there exists a constant $C>0$ such that $\int_1^R\Phi(t)dt\leq C R,$
for all $R>1$ then (see \cite{Sj})
\begin{equation}\liminf_{t\to \infty}\Phi(t)<\infty. \label{calculus}
\end{equation}
The next result  can be considered as an $L^p$ version of Theorem
\ref{iobs} (a).
\begin{theorem} Suppose that
$1<p<2$ and $u\in E_{\lambda}$ for some $\lambda=\alpha
+i\gamma_{p'}\rho,$ $\alpha\in\R.$ If $M_{p'}(u)<\infty$
then $u=\mathcal P_{\lambda}f$ for
some $f\in L^{p'}(K/M).$ Moreover $M_{p'}(u)\asymp  \|f\|_{L^{p'}(K/M)}$.
 \label{Mpconjecture}
\end{theorem}
\begin{proof}
We suppose that $M_{p'}(u)<\infty.$ It then follow that there exists
$R_0>1$ such that for all $R\geq R_0,$ $$
\int_1^R\int_K|u(ka_t)|^{p'}dke^{2\rho t}dt\leq C M_{p'}(u)^{p'}R.$$
That is, \begin{equation}
\int_1^R\int_K\left|\frac{u(ka_t)}{\phi_{\lambda}(a_t)}\right|^{p'}dkdt\asymp
\int_1^R\int_K\left|\frac{u(ka_t)}{e^{-2\rho
t/p'}}\right|^{p'}dkdt\leq CM_{p'}(u)^{p'}R.\label{polar-3} \end{equation}
 By (\ref{calculus})  there
exists a sequence $\{t_j\}\to\infty$ such that
\begin{equation}\lim_{t_j\to\infty}\int_K\left|\frac{u(ka_{t_j})}{\phi_{\lambda}(a_{t_j})}\right|^{p'}dk\leq
CM_{p'}(u)^{p'}.\label{unifb}\end{equation} This, in particular, implies that
the sequence of functions
$\left\{\frac{u(\cdot~a_{t_j})}{\phi_{\lambda}(a_{t_j})}\right\}, j=1,2\dots, \infty$
is a norm bounded family in $L^{p'}(K/M).$ By Eberlein-\v{S}mulian
theorem there exists a subsequence of $\{t_j\}$ (which we will
continue to call $\{t_j\}$) and an $f\in L^{p'}(K)$ such that \begin{equation}
\lim_{t_j\to\infty}\int_K
\frac{u(ka_{t_j})}{\phi_{\lambda}(a_{t_j})}\psi (k)dk=\int_kf(k)\psi
(k)dk \text{ for all } \psi\in L^p(K).\label{wconvergence}.\end{equation}
It also follow from (\ref{unifb}) that $\|f\|_{L^{p'}(K/M)}\leq
CM_{p'}(u).$ On the other hand, by Lemma \ref{exponential} and
Theorem \ref{lewis}, we  know that $u=\mathcal P_{\lambda}T$ for some
distribution $T$ on $K/M.$ It then follows from Lemma \ref{lemma-Ben} that
\[\lim_{t\to\infty}\int_K
\frac{u(ka_t)}{\phi_{\lambda}(a_t)}\psi (k)dk=T(\psi) \text{ for
all } \psi\in C^{\infty}(K).\] In particular, \begin{equation}
\lim_{t\rightarrow\infty}\int_K
\frac{u(ka_{t_j})}{\phi_{\lambda}(a_{t_j})}\psi
(k)dk=T(\psi) \text{ for all } \psi\in C^{\infty}(K)
\label{limit}\end{equation} From (\ref{wconvergence}) and (\ref{limit}) now we have
$$T(\psi)=\int_Kf(k)\psi (k)dk, \text{ for all } \psi\in C^{\infty}(K)$$
i. e., $T=f$. It now follows easily from (\ref{mike}) that $M_{p'}(u)=M_{p'}(\mathcal P_\lambda f)\le C\|f\|_{L^{p'}(K/M)}$.
\end{proof}
\begin{remark} \label{rmk-p-infty}
If $u\in E_\lambda\cap L^{p', \infty}(X)$, for $\lambda$ and $p$ are as in Theorem \ref{Mpconjecture}, then by Corolary \ref{lr}, $u=\mathcal P_\lambda f$ for some $f\in L^{p'}(K/M)$. Therefore by Lemma \ref{mp} and Theorem \ref{Mpconjecture} $\|f\|_{L^{p'}(K/M)}\le C \|u\|_{p', \infty}$. This along with (\ref{poisson-estimate-X}) shows
 $\|u\|_{p', \infty}\asymp \|f\|_{L^{p'}(K/M)}$.
\end{remark}
We offer  another characterization of eigenfunctions which
generalizes Theorem \ref{sos} for  $-\rho<\Im \lambda<0.$
\begin{theorem}
Let $1< p<2,$ $1<q<\infty$ and $u\in E_{\lambda}$ with
$\lambda=\alpha+i\gamma_{p'}\rho.$ Then $u=\mathcal P_\lambda f$ for some
$f\in L^q(K/M)$ if and only if $\mathcal
A_{p',q}(u)<\infty$.\label{lohoue-knapp-p}
\end{theorem}
\begin{proof}
We  assume that $\mathcal A_{p',q}(u)<\infty$, i.e.,
\begin{equation} s{\bf \mathcal A}_q(u)^*(s)^{p'}\le \mathcal A_{p',
q}(u)^{p'},\text{ for all } s>0.\label{mq-est4} \end{equation}

By Lemma \ref{exponential},  we have that for all $k\in K$ and
$t>0,$
$$|u(ka_t)|\le C \mathcal
A_{p',q}(u)e^{\alpha t}$$ which in particular implies by Theorem \ref{lewis} that $u=\mathcal P_\lambda T$ for some distribution $T$ on $K/M$.
Taking $L^p$-norm on $K$ of both sides of the inequality above we get for all $t>0$
\begin{equation} \mathcal A_q(u)(a_t)\le C \mathcal A_{p', q}(u) e^{\alpha t}.
\label{mq-est}\end{equation} Using  estimate of $\phi_{i\gamma_p\rho}$ (see
section 2) we have for all $R>1$
\begin{eqnarray}\int_1^R\left(\int_K\left|\frac{u(ka_t)}{\phi_{i\gamma_{p'}\rho(a_t)}}\right|^qdk\right)^{p'/q}dt
&\asymp& \int_{E_R}\mathcal A_q(u)(x)^{p'}dx\nonumber\\
&\leq&C\int_0^{\infty}(\chi_{E_R}\mathcal A_q(u))^*(s)^{p'}ds
\label{mq-est2},\end{eqnarray} where $E_R=\{x\in X\mid 1\leq |x|\leq R\}.$
As $\chi_{E_R}^*=\chi_{[0,|E_R|)}$ it follows from (\ref{mq-est})
that
\begin{equation} (\chi_{E_R}\mathcal A_q(u))^*(s)^{p'}\leq C {\bf \mathcal
A}_{p', q}(u)^{p'} e^{p'\alpha R}\chi_{[0,|E_R|)}(s), \text{ for
all } s>0. \label{mq-est3}\end{equation} Hence, as in the proof of Theorem
\ref{Mpconjecture}, it follows from (\ref{mq-est2}), (\ref{mq-est3})
and (\ref{mq-est4}) that
\begin{eqnarray}
\int_1^R\left(\int_K\left|\frac{u(ka_t)}{\phi_{i\gamma_{p'}\rho(a_t)}}\right|^qdk\right)^{p'/q}dt&\leq&
C\mathcal
A_{p', q}(u)^{p'}\int_0^{|E_R|}\min \left\{\frac{1}{t},e^{p'\alpha R}\right\}dt\nonumber\\
&\leq &C_{\alpha,p}\mathcal A_{p', q}(u)^{p'}R\label{mq-est5} \end{eqnarray}
By  (\ref{calculus}) there is a sequence $\{t_j\}\rightarrow \infty$
such that,
$$\lim_{j\rightarrow\infty}\left(\int_K\left|\frac{u(ka_t)}{\phi_{i\gamma_{p'}\rho(a_t)}}\right|^qdk\right)^{1/q}
\leq C\mathcal A_{p', q}(u).$$ Hence by Lemma \ref{lemma-Ben}  we have
\begin{equation}
\lim_{j\rightarrow\infty}\int_K\frac{u(ka_t)}{\phi_{i\gamma_p\rho}(a_t)}\psi(k)dk=T(\psi), \text{ for all } \psi\in C^{\infty}(K/M) \label{helgason}
\end{equation} for some distribution $T$ on $K/M$.   As in Theorem \ref{Mpconjecture}, an application of
Eberlein-\v{S}mulian theorem  shows that there exists $f\in
L^q(K/M)$ such that $u=\mathcal P_\lambda f$ and
$$\|f\|_{L^q(K/M)}\leq C \mathcal
A_{p', q}(u).$$
Using (\ref{mike}) and the estimate of the maximal function $\wtilde{f}$ mentioned there we have,
$$\mathcal A_q(u)(a_t) = \left(\int_{K/M}\left|\mathcal P_{i\gamma_{p'}\rho}f(ka_t)
\right|^q\right)^{1/q}\leq C
\|f\|_{L^q(K/M)}\phi_{i\gamma_{p'}\rho}(a_t).$$ The  estimate
of $\phi_{i\gamma_p\rho}$ (see section 2) then shows that $\mathcal A_q(u)\in
L^{p',\infty}(X).$
\end{proof}
\begin{remark}
It is clear that for $1<p<2$, $1<q<\infty$ and $\lambda=\alpha-i\gamma_p\rho$,
$\alpha\in\R$, there exists a positive constant $C$ such that for all
$u$ we have $\mathcal A_{p',q}(u)\leq C\sup_{\{a_t \mid
t>0\}}\phi_{i\gamma_p\rho}(a_t)^{-1}\mathcal A_q(u)(a_t)$. Thus it
follows that Theorem \ref{lohoue-knapp-p} generalizes Theorem
\ref{iobs} for these values of $\lambda$.
\end{remark}

In the rest of the section we shall consider only
$\lambda\in\R^\times$. We know from
Corollary \ref{lr} that if $1<p<2$ and $\Im\lambda=\gamma_{p'}\rho$
then $\mathcal P_{\lambda}f\in L^{p',\infty}(X)$ whenever $f\in L^{p'}(K/M).$
It is not known to us  whether $\lambda\in\R^\times$ and
$f\in L^2(K/M)$ ensure that the Poisson transform $\mathcal P_{\lambda}f\in
L^{2,\infty}(X).$ However the following result shows that the
converse is true.

\begin{theorem}
If $u\in E_{\lambda}\cap L^{2,\infty}(X)$ for some
$\lambda\in\R^\times$ then $u=\mathcal P_{\lambda}f$ for some $f\in
L^2(K/M).$\label{weakl2case}
\end{theorem}
\begin{proof}
By Lemma \ref{exponential} and Lemma \ref{mp} it follows that
$M_2(u)\leq C\|u\|_{L^{2,\infty}(X)}.$ Applying Theorem \ref{iobs}
(a), we conclude that $u=\mathcal P_{\lambda}f$ for some $f\in L^2(K/M).$
\end{proof}

The last  result of this section  will be restricted  to the hyperbolic spaces over $\R,$ $\C,$ or
$\mathbb H.$ We will need  the following  results proved in
\cite[Theorem B and Theorem A, (ii)]{Bou-Sami}:

\noindent{\bf (I)} If $1<p<\infty$ and $\lambda\in\R^\times$
then there exists a positive constant $C$ such that for all $f\in
L^p(K/M),$ the following estimate holds, \begin{equation}
\mbox{sup}_{t>0}~e^{\rho
t}\left(\int_{K/M}|\mathcal P_{\lambda}f(ka_t)|^pdk\right)^{1/p}\leq
C_p\|f\|_{L^p(K/M)}.\label{sami1} \end{equation}

\noindent{\bf (II)} Let $f\in L^2(K/M)$ and
$\lambda\in\R^\times$. We define $f_R(b)=\frac 1R
|\hc(\lambda)|^{-2}\int_{B(0,R)}\mathcal P_{\lambda}f(x)e^{(i\lambda-\rho)H(x^{-1}b)}dx$ for $b\in K/M$.
Then \begin{equation} f_R\longrightarrow  f \text{ in } L^2(K/M)
\text{ as } R\rightarrow\infty. \label{sami2}\end{equation}

\begin{theorem}
Let $1<q<\infty$ and $X$ be hyperbolic space over $\R,\C$ or
$\mathbb H.$ If $u\in E_{\lambda}$ with $\lambda\in\R^\times$
then $u=\mathcal P_{\lambda}f$, for some $f\in L^q(K/M)$ if and only if
$\mathcal A_{2,q}(u)<\infty$\  $($consequently if and only if
$\sup_{t>0}e^{\rho t} \mathcal A_q(u)<\infty$$)$. \label{samib}
\end{theorem}
This  theorem  is slightly general than the conjecture  posed in
\cite{Bou-Sami} (see   the line following  Theorem \ref{iobs}).
\begin{proof} If $u=\mathcal P_\lambda f$  for some $f\in L^q(K/M)$ then by  (\ref{sami1})  $\sup_{t>0}e^{\rho t}\mathcal A_q(u)(a_t)<\infty$,
consequently $\mathcal A_{2,q}(u)\leq C\sup_{t>0}e^{\rho t} \mathcal A_q(u)<\infty$. Therefore to complete the proof it is enough to show that
$\mathcal A_{2,q}(u)<\infty$ implies $u=\mathcal P_\lambda f$ for some $f\in
L^q(K/M)$.

We first consider the case $q=2,$ that is, $u\in E_{\lambda}$ and
$\mathcal A_{2,2}(u)<\infty.$ We note that
\begin{equation}\int_{B(0,R)}|u(x)|^2dx=\int_0^R\int_K|u(ka_t)|^2dkJ(t)dt=\int_{B(0,R)}\mathcal
A_2(u)(x)^2dx.\label{expression}\end{equation} By Lemma \ref{exponential}  we
have $\mathcal A_2(u)(ka_s)\leq \mathcal A_{2, 2}(u) e^{\alpha s}$
for all $k\in K$ and $s>0.$ This implies that for all $R>0$ the
following inequality holds,
$$(\chi_{B(0,R)}\mathcal A_2(u))^*\leq C \mathcal A_{2,2}(u) e^{\alpha R}\chi_{B(0,R)}^*.$$
Using Lemma \ref{exponential}, (\ref{expression}) and the fact
$\mathcal A_2(u)\in L^{2,\infty}(X)$ it follows exactly as in the
proof of Theorem \ref{lohoue-knapp-p} that
\begin{equation}\int_{B(0,R)}|u(x)|^2dx\leq C R \mathcal
A_{2,2}(u)^2.\label{sami4}\end{equation} By Theorem \ref{iobs}
(a),   there exists $f\in L^2(K/M)$ such that $u=\mathcal P_\lambda f.$ Let
us consider the case $q\neq 2,$ that is, $u\in E_{\lambda}$ and
$\mathcal A_q(u)\in L^{2,\infty}(X).$ Let $\{\varphi_n\}\subset
C(K)$ be an approximate identity on $K.$ For each fixed $t>0$ we
define the functions on $K$: $u^{a_t}(k)=u(ka_t),\quad u_n^{a_t}(k)=\varphi_n\ast u^{a_t}(k),
\quad k\in K.$ Since $K$ is a finite measure space, we have for $q\ge 1$,
$\|u_n^{a_t}\|_{L^2(K)}\leq C \|u^{a_t}\|_{L^q(K)}\|\varphi_n\|_{L^2(K)}.$
Consequently, for all $n$ we have $\mathcal A_2(u_n)(a_t)=\left(
\int_K|u_n(ka_t)|^qdk\right)^{1/q}\leq\|\varphi_n\|_{L^2(K)}
\mathcal A_q(u)(a_t).$ So $\mathcal A_2(u_n)\in L^{2,\infty}(X).$
Using the case $q=2$ we conclude that there exists $F_n\in L^2(K/M)$
such that $u_n=\mathcal P_{\lambda}F_n$ . We also have
\begin{equation}\|u_n^{a_t}\|_{L^q(K)}\leq
\|\varphi_n\|_{L^1(K)}\|u^{a_t}\|_{L^q(K)}=\|u^{a_t}\|_{L^q(K)}=\mathcal
A_q(u)(a_t),\label{sami3}\end{equation} as $\int_K\phi_n(k)dk=1$. For $R>0$ we define
$$g_R^n(k)=\frac{1}{R}\int_{B(0,R)}u_n(x)e^{(i\lambda-\rho)H(x^{-1}k)}dx,\,\,  k\in K.$$
By (\ref{sami2}) it follows that for each fixed $n,$
\begin{equation}\lim_{R\rightarrow\infty}|\hc(\lambda)|^{-2}\int_Kg_R^n(k)\overline{\psi(k)}dk=\int_KF_n(k)\overline{\psi(k)}dk,
\text{ for all } \psi\in C(K).\label{convergence}
\end{equation}
For each $n$ we define for $\psi\in C(K)$,
$L_n(\psi)=\int_KF_n(k)\overline{\psi_n(k)}dk.$ We claim   that
$L_n\in L^{q'}(K/M)^*.$ Indeed  for $\psi\in C(K)$
\begin{eqnarray*}
\left|\int_Kg_R^n(k)\overline{\psi(k)}dk\right|&=&\frac{1}{R}\left|\int_K\int_{B(0,R)}u_n(x)e^{(i\lambda-\rho)H(x^{-1}k)}
\overline{\psi(k)}dxdk\right|\\
&=&\frac{1}{R}\left|\int_{B(0,R)}u_n(x)\overline{\mathcal P_{\lambda}\psi(x)}dx\right|\\
&\leq &\frac{1}{R}\int_0^R\left(\int_K|u_n(ka_t)|^qdk\right)^{1/q}\left(\int_K|\mathcal P_{\lambda}\psi(ka_t)|^{q'}\right)^{1/q'}J(t)dt\\
&\leq&\frac{C}{R}\|\psi\|_{L^{q'}(K/M)}\int_0^R\left(\int_K|u_n(ka_t)|^qdk\right)^{1/q}e^{-\rho
t}J(t)dt \text{ (by (\ref{sami1}))}\\
 &\le &\frac{C}{R}\|\psi\|_{L^{q'}(K/M)}\int_0^R\mathcal
A_q(u)(a_t)e^{-\rho
t}J(t)dt \text{ (by (\ref{sami3}))}\\
&\leq &\frac{C}{R}\|\psi\|_{L^{q'}(K/M)}\left(\int_0^R\mathcal
A_q(u)(a_t)^2J(t)dt\right)^{1/2}\left(\int_0^Re^{-2\rho
t}J(t)dt\right)^{1/2}\\
&\leq&\frac{C}{R}\|\psi\|_{L^{q'}(K/M)}\mathcal A_{2,q}(u)R^{1/2}R^{1/2} \text{ (by (\ref{expression}) and (\ref{sami4}))}\\
&=&C\|\psi\|_{L^{q'}(K/M)}\mathcal A_{2,q}(u)
\end{eqnarray*}

From above, using (\ref{convergence})  we  conclude that
$\left|\int_KF_n(k)\overline{\psi(k)}dk\right|\leq
C\|\psi\|_{L^{q'}(K/M)}\mathcal A_{2,q}(u),$ that is, $L_n\in
L^{q'}(K/M)^*$ with its norm as linear functional dominated by
$C\mathcal A_{2,q}(u)$ for all $n.$ By Eberlein-\v{S}mulian theorem
for reflexive spaces and Riesz representation theorem it follows
that there exists a sequence $n_j$ and $F\in L^q(K/M)$ such that
$\|F\|_{L^q(K/M)}\leq C\|\mathcal A_q(u)\|_{L^{2,\infty}(X)}$ and
$$\lim_{j\rightarrow\infty}\int_KF_{n_j}(k)\overline{\psi(k)}dk=\int_KF(k)\overline{\psi(k)}dk,
\text{ for all } \psi\in L^{q'}(K/M).$$ Since for each fixed
$x\in X$ the function $k\mapsto e^{(i\lambda-\rho)H(x^{-1}k)}$ is in
$L^{q'}(K/M)$ it follows that,

$$\mathcal P_{\lambda}F(x)=\lim_{j\rightarrow\infty}\int_KF_{n_j}(k)e^{-(i\lambda+\rho)H(x^{-1}k)}dk\nonumber\\
= \lim_{j\rightarrow\infty}u_{n_j}(x)
=\lim_{j\rightarrow\infty}\varphi_{n_j}\ast u^{a_t}(k_1)
=u(k_1a_t),$$ when $x=k_1a_t$.
 This
completes the proof.\end{proof}

\section{Roe's theorem for tempered distributions on DR spaces}
This section is one of the two technical hearts of the paper. Theorems A and B and their various analogues will use the two theorems stated and proved in this section. It was mentioned in the introduction that the DR spaces have eigenfunctions (such as $x\mapsto \wp_\lambda(x,n)$ on $S$ or $x\mapsto e^{(i\lambda+\rho)H(x^{-1}k)}$ on $X$) which are not in any Lebesgue or Lorentz spaces. However they are $L^p$-tempered distributions when $\lambda\in S_p$ (see Lemma~\ref{Mp-Apq-dist}~(b) below) and hence theorems of this section can accommodate them. (See also subsection 6.4.)

We recall that $R$ denotes the  radialization operator on the DR space $S$ (see section
2). Let $T$ be a  $L^p$-tempered distribution for a fixed $p\in (1, 2]$. The distribution $T$ is called
radial if
$$\langle T, \psi\rangle=\langle T, R(\psi)\rangle, \text{ for all } \psi\in C^p(S).$$
In general the radial part
$R(T)$ of a $L^p$-tempered distribution $T$ is  a $L^p$-tempered
distribution defined by
$$\langle R(T), \psi\rangle=\langle T, R(\psi)\rangle, \text{ for all } \psi\in C^p(S).$$ Thus when $T$ itself is radial then $T=R(T)$.
We shall say that  $T$ has {\em no
radial part} if $R(T)=0$, in other words, if $\langle T,
\psi\rangle=0$ for all $\psi\in C^p(S)^\#$.

Left translation $\ell_x$ of $T$ by an element $x\in S$  is defined
by the following: for $\psi\in C^p(S)$, let
$\psi^\ast(x)=\psi(x^{-1})$.  Then
$$\langle\ell_x T, \psi\rangle=T(\ell_{x^{-1}}\psi)=T\ast \psi^\ast(x^{-1}).$$  If $\psi$ is radial then $\psi^\ast(x)=\psi(x)$ and hence  $\langle \ell_x T, \psi\rangle=T\ast \psi(x^{-1})$. For a radial $L^p$-tempered distribution $T$, its spherical Fourier transform $\what{T}$ is defined as a linear functional on $C^p(\what{S})^\#$ by the following rule: $$\langle \what{T}, \phi\rangle=\langle T, \phi^\vee\rangle, \text{ where } \phi\in C^p(\what{S})^\#, \phi^\vee\in C^p(S)^\#\text{ and } \what{\phi^\vee}=\phi.$$

\subsection{Result for $L^2$-tempered distributions}
\begin{theorem} \label{thm-L^2-temp}
Let  $\{T_k\}_{k\in \Z}$ be a doubly infinite sequence of
$L^2$-tempered distributions on $S$ satisfying

{\em (i)} $\Delta T_k=z T_{k+1}$  for some  nonzero  $z\in \C$ and

{\em (ii)} for all $\psi\in C^2(S)$,  $|\langle T_k,
\psi\rangle|\le M \gamma(\psi)$ for some fixed seminorm $\gamma$
of $C^2(S)$ and $M>0$.

Then
\begin{enumerate}
\item[(a)] $|z|\ge \rho^2$ implies that  $\Delta T_0=-|z| T_0$ and
\item[(b)] $|z|<\rho^2$ implies $T_k=0$ for all $k\in \Z$.
\end{enumerate}
\end{theorem}

\begin{proof}
We shall divide the proof of (a) in
two parts. In the first part we shall prove the assertion with the
extra assumption that the distributions $T_k$ are radial.

\noindent {\em  Case 1:  $T_k$ are radial.} The argument in this part is close to one used in \cite{Howd-Reese}.
We shall further divide the proof for this case in a few steps.

\noindent{\bf Step 1.} We write $z=(\alpha^2+\rho^2)e^{i\theta}$
where $|z|=\alpha^2+\rho^2$ for some $\alpha\ge 0$ and $\theta
=\arg z$.
In this step, we shall show that the distributional support of
$\what{T_0}$ is $\{\alpha, -\alpha\}.$

It follows from hypothesis (i) of the theorem that  $\Delta^k
T_0=e^{ik\theta}(\alpha^2+\rho^2)^kT_k$ for all $k\in \Z$. This
implies $$\what{T_0}=(-1)^k
e^{ik\theta}\left(\frac{\alpha^2+\rho^2}{\lambda^2+\rho^2}\right)^k\what{T_k},$$ where $\lambda$ is a dummy variable.
Let $\phi\in C^2(\what{S})^\#$ be such that $\phi(\lambda)=0$ if
$0\le \lambda<\alpha+\varepsilon$. From the description of
$C^2(\what{S})^\#$ this implies that $\phi(\lambda)=0$ if
$|\lambda|<\alpha+\varepsilon$. We claim that $\langle\what{T_0},
\phi\rangle=0$.

Let $\psi\in C^2(S)^\#$ be the pre-image of $\phi(\lambda)
(\alpha^2+\rho^2)^k/(\lambda^2+\rho^2)^k$.
Then from hypothesis (i) and (ii) we get,
\begin{equation*}
|\langle\what{T_0}, \phi\rangle|=|\langle \what{T_k},
e^{ik\theta}\left(\frac{\alpha^2+\rho^2}{\lambda^2+\rho^2}\right)^k
\phi\rangle| =\left|\left\langle T_k, \psi \right\rangle\right | \le
M \gamma(\psi) \le M \mu_{\beta,
\tau}\left[\left(\frac{\alpha^2+\rho^2}{\lambda^2+\rho^2}\right)^k
\phi\right ],
\end{equation*}
where
 $$\mu_{\beta, \tau}\left[\left(\frac{\alpha^2+\rho^2}{\lambda^2+\rho^2}\right)^k\phi\right]=
\sup_{|\lambda|>\alpha+\varepsilon}(1+|\lambda|)^\beta
|\frac{d^\tau}{d\lambda^\tau}
\left(\frac{\alpha^2+\rho^2}{\lambda^2+\rho^2}\right)^k
\phi(\lambda)|,$$ for some positive integers $\beta$ and $\tau$.
It is now easy to verify that as $k\rightarrow +\infty$,
$$\sup_{|\lambda|>\alpha+\varepsilon}(1+|\lambda|)^\beta |\frac{d^\tau}{d\lambda^\tau}
\left(\frac{\alpha^2+\rho^2}{\lambda^2+\rho^2}\right)^k
\phi(\lambda)|\rightarrow 0.$$
Again   a similar argument taking $k\rightarrow -\infty$ will show
that $\langle\what{T_0}, \phi\rangle=0$ if $\phi(\lambda)=0$ for
$|\lambda|>\alpha-\varepsilon$. This establishes the claim.
\vspace{.2in}

\noindent{\bf Step 2.} In this step we shall show that
\begin{equation}\label{generalized-eigen}(\Delta+\alpha^2+\rho^2)^{N+1}T_{0}=0, \text{ equivalently, } (\alpha^2-\lambda^2)^{N+1}\what{T_0}=0,\text{ for some }N\in \Z^+.\end{equation}

Let $g$ be an even function in $C_c^\infty(\R)$ such that $g\equiv
1$ on $[-1/2, 1/2]$ and support of $g$ is contained in $(-1, 1)$.
Let $g_r$ be defined by $g_r(\xi)=g(\xi/r)$. For a function $\phi\in
C^2(\what{S})^\#$ we define
$$H_r(\lambda)=(\alpha^2-\lambda^2)^{N+1}g_r(\alpha^2-\lambda^2)\phi(\lambda).$$ It is clear that $H_r\in C^2(\what{S})^\#$. Let $h_r\in C^2(S)^\#$ be the pre-image of $H_r$.

We fix $\varepsilon >0$ and $r(\varepsilon)=3\varepsilon(2\alpha+\varepsilon)$.  It is clear that for $\lambda\in (\alpha-\varepsilon, \alpha+\varepsilon)$ we have $g_{r(\varepsilon)}(\alpha^2-\lambda^2)\equiv 1$. Using the fact that
 the support of $\what{T_0}$ is $\{-\alpha, \alpha\}$ and $\what{T_0}$ acts only on even functions,  we get

\begin{eqnarray*}
|\langle (\alpha^2-\lambda^2)^{N+1}\what{T_0}, \phi\rangle|= |\langle\what{T_0}, (\alpha^2-\lambda^2)^{N+1}\phi\rangle|&=&|\langle\what{T_0}, H_{r(\varepsilon)}\rangle|=|\langle T_0, h_{r(\varepsilon)}\rangle|\\
&\le& M\gamma(h_{r(\varepsilon)})\le M\mu_{\beta, \tau}(H_{r(\varepsilon)}),
\end{eqnarray*}
for some positive integers $\beta, \tau$. The proof of this step
will be completed if we show that
\begin{equation}\sup_{\lambda\in \R}(1+|\lambda|)^\beta \left| \frac{d^\tau}{d\lambda^\tau}((\alpha^2-\lambda^2)^{N+1} g_{r(\varepsilon)}(\alpha^2-\lambda^2)\phi(\lambda)\right|
\label{**}
 \end{equation} converges to zero as $\varepsilon\rightarrow 0$ for a suitably large $N$.

We note that $|d^s/d\lambda^s\phi(\lambda)|\le M$ for some $M>0$ for all
$s\le \tau$, $|d^s/d\lambda^s g_{r(\varepsilon)}(\lambda)|\le B/|\lambda|^s$ for some $B>0$
for all $s\le \tau$ and the function
$G_{r(\varepsilon)}(\lambda)=g_{r(\varepsilon)}(\alpha^2-\lambda^2)$ along with its derivatives
vanishes when $|\alpha^2-\lambda^2|>{r(\varepsilon)}$.

Therefore in the expression  (\ref{**}) the supremum can be taken
over $\{\lambda\in \R\mid |\alpha^2-\lambda^2|<{r(\varepsilon)}\}$ and thus  it is
dominated by $C_\phi |\alpha^2-\rho^2|^\beta\le C_\phi {r(\varepsilon)}^\beta$ for
some $\beta>0$ and hence converges to zero as $\varepsilon\rightarrow 0$.

This shows that $|\langle (\alpha^2-\rho^2)^{N+1}\what{T_0},
\phi\rangle|=0$ which completes Step 2. \vspace{.2in}

\noindent{\bf Step 3:} We shall establish  that
$$(\Delta+\alpha^2+\rho^2)T_{0}=0.$$ From
(\ref{generalized-eigen}) it follows that $$\mathrm{Span} \{T_{0},
T_{1}, \cdots\}=\mathrm{Span} \{T_{0}, \Delta T_{0}, \cdots,
\Delta^N T_{0} \}=\mathrm{Span} \{T_{0},  T_{1}, \cdots, T_{N}\}.$$
Suppose that $(\Delta+\alpha^2+\rho^2)T_{0}\neq 0$. Let $k_0$ be the
largest positive integer such that $(\Delta+\alpha^2+\rho^2)^{k_0}
T_{0}\neq 0$. Then $k_0\le N$.

Let $g=(\Delta+\alpha^2+\rho^2)^{k_0-1} T_{0}\in \text{ Span }
\{T_{0},  T_{1}, \cdots, T_{N}\}$. We assume that $g=a_0
T_{0}+\dots+a_N T_{N}$. Then
\begin{align}(\Delta+\alpha^2+\rho^2)^2g=(\Delta+\alpha^2+\rho^2)^{k_0+1}T_{0}=0&,
(\Delta+\alpha^2+\rho^2)g=(\Delta+\alpha^2+\rho^2)^{k_0}T_{0}\neq
0. \label{*}
\end{align}
Using binomial expansion and (\ref{*}) we get for any positive
integer $k$,
\begin{eqnarray*}
\Delta^k g &=&(\Delta+(\alpha^2+\rho^2)-(\alpha^2+\rho^2))^k g\\
&=& k
(-1)^{k-1}(\alpha^2+\rho^2)^{k-1}(\Delta+(\alpha^2+\rho^2))g+(-1)^k(\alpha^2+\rho^2)^k
g.
\end{eqnarray*}
This implies for any $\psi\in C^2(S)^\#$,
\begin{equation}|\langle (\Delta+(\alpha^2+\rho^2))g, \psi\rangle|\le \frac 1k (\alpha^2+\rho^2)^{1-k}|\langle\Delta^k g, \psi\rangle|+\frac 1k (\alpha^2+\rho^2)|\langle g, \psi\rangle|.
\label{step4-1}
\end{equation}

Since,
\begin{eqnarray*}|\langle\Delta^k g , \psi\rangle | &=&|\langle\Delta^k(a_0 T_{0}+a_1T_{1}+\cdots+a_N T_{N}), \psi\rangle|\\
&=&|a_0(\alpha^2+\rho^2)^k e^{ik\theta} \langle T_{k} , \psi\rangle+\cdots+ a_N(\alpha^2+\rho^2)^k  e^{ik\theta}\langle T_{N+k} , \psi\rangle|\\
&\le& (\alpha^2+\rho^2)^k|a_0 \langle T_{k}, \psi\rangle|+\cdots+ |a_N \langle T_{N+k}, \psi\rangle|\\
&\le& M(\alpha^2+\rho^2)^k(|a_0|+\cdots+|a_{N}|)\gamma(\psi),
\end{eqnarray*}
from above and  by (\ref{step4-1}) it follows that,
$$|\langle(\Delta+(\alpha^2+\rho^2)) g, \psi\rangle|\le M\frac{(\alpha^2+\rho^2)}k(|a_0|+\cdots+|a_{N}|)\gamma(\psi)+\frac{(\alpha^2+\rho^2)}k |\langle g, \psi\rangle|$$ and the right side goes to $0$ as $k\rightarrow \infty$. Therefore by (\ref{*})
 $(\Delta+(\alpha^2+\rho^2))^{k_0} T_{0}=0$ which contradicts the assumption on $k_0$. Thus we have shown that $N=0$, i.e., $(\Delta+\alpha^2+\rho^2)T_{0}=0.$
 This completes the  proof of (a) for radial $T_k$.
\vspace{.2in}

Now we shall withdraw the assumption of radiality from the sequence
$\{T_k\}$.

\noindent {\em Case 2: $T_k$ are not necessarily radial.}  In this  proof we shall frequently use the fact that $\Delta$ commutes with the radialization   and with translations (see section 2). The following steps will lead to the proof.

\noindent{\bf Step $1'$}. Given any $k\in \Z$,  there  is a $x\in S$ such that
$\ell_x T_k$ has nonzero radial part. Indeed if $R(\ell_xT_k)=0$  for all $x\in S$, then
$\langle\ell_x T_k, h_t\rangle=0$ for all $t>0$ where $h_t$ denotes the heat kernel, which is a radial function (see \cite{ADY}). That is $T_k\ast
h_t\equiv 0$. But  $T_k\ast h_t\rightarrow T_k$ as $t\rightarrow 0$ in the sense of distribution.
Therefore $T_k=0$ and there is nothing to prove. We note that
this also shows that if for two $L^2$-tempered distribution $T$ and
$T'$, $R(\ell_x T)=R(\ell_xT')$ for all $x\in S$, then $T=T'$.
\vspace{.1in}

\noindent{\bf Step $2'$}. We claim that if  $R(\ell_y T_0)\neq 0$ for
some $y\in S$, then $R(\ell_y T_k)\neq 0$ for all $k\in \Z$. It is
enough to  show that if   $R(\ell_y T_0)\neq 0$ then $R(\ell_y T_{-1})\neq 0$ and $R(\ell_y T_1)\neq
0$. Indeed if $R(\ell_y T_{-1})= 0$ then
$\Delta R(\ell_y T_{-1})= 0$ which implies $R(\ell_y T_0)= 0$ as
$\Delta T_{-1}=zT_0$ for $z\neq 0$.

If  $R(\ell_y T_1)= 0$, then $\langle
\ell_y T_1, \psi\rangle=0$ for all $\psi\in C^2(S)^\#$. That is
$\langle \ell_y \Delta T_0, \psi\rangle=0$
and hence $\langle \ell_y T_0, \Delta
\psi\rangle=0$. Since  for any $\phi\in C^2(S)^\#$,
$\what{\phi}(\lambda)(\lambda^2+\rho^2)^{-1}\in C^2(\what{S})^\#$ (see section 2),
$\phi$ can be written as $\phi=\Delta\psi$ for some $\psi\in
C^2(S)^\#$. Thus  $\langle \ell_y T_0,
\phi\rangle=0$ for any $\phi\in C^2(S)^\#$, i.e.  $R(\ell_y T_0)=0$. \vspace{.1in}

\noindent{\bf Step $3'$}. In this step we shall show that for any $y\in S$, the sequence $\{R\ell_y T_k\}$ of radial distributions satisfies the hypothesis of the theorem.
Since $\Delta$ commutes with radialization and translations it follows from the hypothesis
$\Delta T_k=z T_{k+1}$ that $\Delta R(\ell_y T_k)=z R(\ell_y T_{k+1})$.

It now remains to show that for the seminorm $\gamma$ of $C^2(S)$ in the hypothesis of the theorem  and $\psi_1\in C^2(S)^\#$, $|\langle R(\ell_y T_k), \psi_1\rangle|\le C_y M\gamma(\psi_1)$.
Using the estimate  $\phi_0(x)^{-1}(1+|x|)\asymp
e^{\rho|x|}$ we have for any seminorm $\gamma$ of $C^2(S)$ and $\psi\in C^2(S)$,
\begin{eqnarray*}
\gamma(\ell_y\psi)&=&\sup_{x\in S}
|D\psi(y^{-1}x)|\phi_0(x)^{-1}(1+|x|)^L\\&=&
\sup_{x\in S} |D\psi(x)|\phi_0(yx)^{-1}(1+|yx|)^L\\
&\asymp& \sup_{x\in S} |D\psi(x)|e^{\rho|yx|}(1+|yx|)^{L-1}\\
&\le& e^{\rho|y|}(1+|y|)^{L-1}\sup_{x\in S} |D\psi(x)|e^{\rho|x|}(1+|x|)^{L-1}\\
&\asymp&e^{\rho|y|}(1+|y|)^{L-1}\sup_{x\in S} |D\psi(x)|\phi_0(x)^{-1}(1+|x|)^{L}\\
&=&C_y\gamma(\psi),
\end{eqnarray*}
where the constant $C_y$ depends only on $y\in S$.
Since  $|\langle T_k, \psi\rangle|\le
M\gamma(\psi)$ for any $\psi\in C^2(S)$, it follows that  for  $\psi_1\in
C^2(S)^\#$,
\begin{eqnarray*}
|\langle R(\ell_y T_k), \psi_1\rangle|=|\langle\ell_yT_k,
\psi_1\rangle|=|\langle T_k, \ell_{y^{-1}}\psi_1\rangle|\le
M\gamma(\ell_{y^{-1}}\psi_1)\le C_y M \gamma(\psi_1).
\end{eqnarray*}

\noindent{\bf Step $4'$}. From the previous step and the result proved for radial distributions we conclude that
$$\Delta R(\ell_y(T_0))=-|z| R(\ell_y(T_0))\text{  for all } y\in S.$$
(Note that if  $R(\ell_y(T_0))=0$ for some $y\in S$, then  the identity    $\Delta R(\ell_y(T_0))=-|z| R(\ell_y(T_0))$ is trivial.)
Again appealing to the fact that
$\Delta$ commutes with translations and radialization we have
$R(\ell_y(\Delta T_0))=R(\ell_y(-|z| T_0))$  for all $y\in S$. By Step $1'$
 this implies $\Delta T_0=-|z|T_0$ which is the assertion of
(a).

To prove part (b) of the theorem we note again that for an $L^2$-tempered
distribution $T$,  $T=0$ is equivalent to  $R(\ell_xT)=
0$ for all $x\in S$ and  that $\Delta T=\alpha T$ implies
$\Delta R(\ell_xT)= \alpha R(\ell_xT)$ for all $x\in S$. Therefore
it is enough to  assume that $T_k$ are radial. We can proceed as in
the proof of (a) (for radial distributions). For $\phi\in
C^2(\what{S})^\#$ and suitable seminorms $\gamma$ and $\mu$ of $C^2(S)$ and $C^2(\what{S})^\#$ respectively,
\begin{eqnarray*}
|\langle \what{T_0}, \phi\rangle|=\left|\left\langle \what{T_k}, \frac{z^k}{(\lambda^2+\rho^2)^k} \phi\right\rangle\right|
=\left|\left\langle T_k, \left(\frac{z^k}{(\lambda^2+\rho^2)^k} \phi\right)^\vee\right\rangle\right|
&\le& M \gamma\left[\left(\frac{z^k}{(\lambda^2+\rho^2)^k} \phi\right)^\vee\right]\\
&\le& M \mu\left[\frac{z^k}{(\lambda^2+\rho^2)^k} \phi\right].
\end{eqnarray*}

Since $|z|<|\lambda^2+\rho^2|$ for all $\lambda\in S_2=\R$, the
right side goes to $0$ as $k\rightarrow \infty$ and  we conclude
that $\langle T_0, \psi\rangle=0$ for all $\psi\in C^2(S)^\#$.
\end{proof}

\begin{remark} For the symmetric spaces $X$, since $\Delta$ preserves the $K$-types and since  $L^p$-Schwartz space isomorphism theorems are available for functions on $X$ with a single left $K$-type (see e.g. \cite{J-S}), we can decompose  the  distribution $T_k$ in $K$-isotypic components and work restricting to one isotypic component at a time exactly the same way we worked with  the radial functions in the previous theorem. This is an alternative way to prove Theorem \ref{thm-L^2-temp}.   However, this method has no relevance for a general $NA$ group, where the rotation group $K$ is not available.
\label{alternative-for-X}
\end{remark}

\subsection{Result for $L^p$-tempered distributions for $1<p<2$}
A distinguishing feature  of the corresponding theorem for the
$L^p$-tempered distributions is that here a one-sided sequence of
functions will be enough.  Handling this case becomes technically
more challenging  because  the $L^p$-spectrum of  $\Delta$ is a parabolic region in
the complex plane. Thus  one has to enter into the realm of
holomorphic functions and the main technique in the proof of the previous theorem, namely
the use of functions whose Fourier transforms  are ``supported outside a set of positive measure'' will not work.

\begin{theorem}\label{thm-Lp-temp}
For $1<p<2$, let  $\{T_k\}_{k\in \Z^+}$ be an infinite sequence of $L^p$-tempered
distributions on $S$ satisfying

{\em (i)} $\Delta T_k=z T_{k+1}$  for some  nonzero  $z\in \C$ and

{\em (ii)} for all $\psi\in C^p(S)$, let $|\langle T_k,
\psi\rangle|\le M \gamma(\psi)$ for some fixed seminorm $\gamma$
of $C^p(S)$ and $M>0$. Then
\begin{enumerate}
\item[(a)] $|z|= 4\rho^2/pp'$ implies that  $\Delta T_0=-|z| T_0$,
\item[(b)] $|z|<4\rho^2/pp'$ implies that $T_k=0$ for all $k\in \Z$ and
\item[(c)] there are solution which are not eigendistributions  when $|z|>4\rho^2/pp'$.
\end{enumerate}
\end{theorem}

\begin{proof}
(a) It suffices to prove the theorem  with the extra assumption that  $T_k$ are radial
as it is possible to extend the theorem from the radial to the general case using  the  argument given in the proof of Theorem \ref{thm-L^2-temp}

For notational convenience in this proof let us write $\alpha$ for
$i\gamma_p\rho$. Then $|z|=\alpha^2+\rho^2$. We shall show that  $(\alpha^2-\lambda^2)^{N+1} \what{T_0}=0$ for a fixed $N\in \Z^+$.
Since $\Delta^k T_0=z^k T_k$ and  for any $\lambda\in S_p$,
$\lambda^2+\rho^2\neq 0$, we have
$\what{T_0}=z^k (-1)^k(\lambda^2+\rho^2)^{-k}
\what{T_k}$  and hence for a fixed $\phi\in C^p(\what{S})^\#$,
\begin{eqnarray*}
|\langle (\alpha^2-\lambda^2)^{N+1} \what{T_0}, \phi\rangle|&=&|\langle\what{T_0}, (\alpha^2-\lambda^2)^{N+1}\phi\rangle|\\
&=&|\langle \what{T_k}, \left(\frac{\alpha^2+\rho^2}{\lambda^2+\rho^2}\right)^k (\alpha^2-\lambda^2)^{N+1}\phi\rangle|\\
&=&\left|\left\langle T_k, \left(\left(\frac{\alpha^2+\rho^2}{\lambda^2+\rho^2}\right)^k (\alpha^2-\lambda^2)^{N+1}\phi\right)^\vee\right\rangle\right|\\
&\le& M  \gamma\left[\left(\left(\frac{\alpha^2+\rho^2}{\lambda^2+\rho^2}\right)^k (\alpha^2-\lambda^2)^{N+1}\phi\right)^\vee\right]\\
&\le& M\, \mu\left[\left(\frac{\alpha^2+\rho^2}{\lambda^2+\rho^2}\right)^k
(\alpha^2-\lambda^2)^{N+1}\phi\right],
\end{eqnarray*}
where $\mu$ is given by
$$\mu\left[\left(\frac{\alpha^2+\rho^2}{\lambda^2+\rho^2}\right)^k (\alpha^2-\lambda^2)^{N+1}\phi(\lambda)\right]
=\sup_{\lambda\in S_p}  \left|\frac {d^\tau}{d\lambda^\tau}
P(\lambda) \left(\frac{\alpha^2+\rho^2}{\lambda^2+\rho^2}\right)^k
(\alpha^2-\lambda^2)^{N+1}\phi(\lambda)\right|,$$
for some even polynomial $P(\lambda)$ and derivative of even order
$\tau$. We shall temporarily use the notation $$F^k=\left|\frac
{d^\tau}{d\lambda^\tau} P(\lambda)
\left(\frac{\alpha^2+\rho^2}{\lambda^2+\rho^2}\right)^k
(\alpha^2-\lambda^2)^{N+1}\phi\right|.$$
We shall use the notation  $A_\tau, B_\tau, C_\tau, C_\tau'$ etc. for positive constants which depend only on $\tau$.

Our  aim is to show that $\sup_{\lambda\in S_p}F^k\rightarrow 0$ as $k\rightarrow \infty$.
We note that in the definition of $\mu$ we can take the supremum
only on $S_p^+=\{\lambda\in S_p\mid \Im \lambda\ge 0\}$ as $\phi$ is
even being the image of a radial function.
We also need the following observations: For $\lambda\in S_p^+$,
\begin{equation}\label{lp-dist-1}
\left|\frac{\alpha^2+\rho^2}{\lambda^2+\rho^2}\right|\le 1, \,\, \left|\frac{\alpha^2-\lambda^2}{\lambda^2+\rho^2}\right|
=\left|\frac{\alpha^2+\rho^2-(\lambda^2+\rho^2)}{\lambda^2+\rho^2}\right|<1.
\end{equation}

For $\lambda\in S_p^+$ we write  $\lambda=a\rho+ib\rho$. Then  $a\in \R$ and $0\le b \le \gamma_p$ and hence,
\begin{equation}
\left|\frac{\alpha^2+\rho^2}{\lambda^2+\rho^2}\right|
\le \frac{1-\gamma_p^2}{1+a^2-b^2}.
\label{1-p}
\end{equation}
We  fix $N =6\tau+1$.
Let $$A_\tau=\max_{i=1}^\tau\{\sup_{\lambda\in S_p^+} |\frac{d^i}{d\lambda^i}(\alpha^2-\lambda^2)^{N+1}P(\lambda)\phi(\lambda)|\mid 0\le i\le \tau \}.$$

An explicit computation shows that for $\lambda\in S_p^+$.
\begin{equation}\label{lp-dist-2}|\frac{d^i}{d\lambda^i}\left(\frac{\alpha^2+\rho^2}{\lambda^2+\rho^2}\right)^{k}|\le B_\tau \left |\frac{\alpha^2+\rho^2}{\lambda^2+\rho^2}\right|^{k} k(k+1)\dots (k+i-1),\ \ \ 0\le i\le \tau.
\end{equation}

Therefore
\begin{eqnarray}&&\left|\frac{d^\tau}{d\lambda^\tau}\left(\frac{\alpha^2+\rho^2}{\lambda^2+\rho^2}\right)^{k+N+1}
(\alpha^2-\lambda^2)^{N+1}P(\lambda)\phi(\lambda)\right|\nonumber\\
&&\le\sum_{i=0}^\tau {\tau\choose i} B_\tau \left|\frac{\alpha^2+\rho^2}{\lambda^2+\rho^2}\right|^{k+N+1}
(k+N+1)\dots (k+N+i) \left|\frac{d^{\tau-i}}{d\lambda^{\tau-i}}(\alpha^2-\lambda^2)^{N+1}P(\lambda)\phi(\lambda)\right| \label{explicit}\\
&&\le A_\tau C_\tau k^\tau \left|\frac{\alpha^2+\rho^2}{\lambda^2+\rho^2}\right|^{k}\left| \frac{\alpha^2-\lambda^2}{\lambda^2+\rho^2}\right|^{N+1}\nonumber\\
&&\le A_\tau C_\tau k^\tau \left|\frac{\alpha^2+\rho^2}{\lambda^2+\rho^2}\right|^{k}. \label{explicit2}
\end{eqnarray}

For any $k=1,2,\dots $ we define a bounded region, $$V_k=\{z\in S_p^+ \mid |\Re z|< k^{-1/4} \text{ and } \gamma_p\rho-\Im z< k^{-1/4}\}.$$

Our strategy is to  show that as $k\rightarrow \infty$,

(i) $\sup_{\lambda\in V_k}F^{k+N+1}\rightarrow 0$
(ii) $\sup_{\lambda\in V_k^c}F^{k+N+1}\rightarrow 0$  uniformly on $k$.

We shall first deal with (ii) above.

We claim that
for all $\lambda\in V_k^c$ and large $k$,
\begin{equation}\left|\frac{d^\tau}{d\lambda^\tau}\left(\frac{\alpha^2+\rho^2}{\lambda^2+\rho^2}\right)^{k+N+1}
(\alpha^2-\lambda^2)^{N+1}P(\lambda)\phi(\lambda)\right|
\le C_\tau' \left(1+\frac c{\sqrt{k}}\right)^{-k} k^\tau,
\label{basic-step-p}
\end{equation} for some constant $c>0$.

In view of  (\ref{explicit2}) it suffices to show that for some constant $c>0$,
\begin{eqnarray*}
\left|\frac{\alpha^2+\rho^2}{\lambda^2+\rho^2}\right|^k\le
\left(1+\frac{c}{\sqrt{k}}
\right)^{-k}.
\end{eqnarray*}
Since $\lambda\in V_k^c$, there are two possibilities: $|a|\ge k^{-1/4}/\rho$ or $\gamma_p-b\ge k^{-1/4}$.

\noindent{\bf Case 1:} $|a|\ge k^{-1/4}/\rho$. Using  $-b^2>-\gamma_p^2$ and (\ref{1-p}),  we get,
\begin{eqnarray*}
\left|\frac{\alpha^2+\rho^2}{\lambda^2+\rho^2}\right|\le
\frac{1-\gamma_p^2}{1+a^2-b^2}\le
\frac{1-\gamma_p^2}{1-\gamma_p^2+a^2}=\left(1+\frac{a^2}{1-\gamma_p^2}\right)^{-1}
\le
\left(1+\frac{k^{-1/2}/\rho^2}{1-\gamma_p^2}\right)^{-1}=\left(1+\frac{c_1}{\sqrt{k}}\right)^{-1},
\end{eqnarray*}
where $c_1=(\rho^{2}-\gamma_p^2\rho^2)^{-1}$.

\noindent{\bf Case 2:} $\gamma_p-b\ge k^{-1/4}$. Then
$\gamma_p^2-b^2=(\gamma_p+b)(\gamma_p-b)\ge \gamma_p  k^{-1/4}$.
Hence
\begin{eqnarray*}
\frac{1-\gamma_p^2}{1+a^2-b^2}\le \frac{1-\gamma_p^2}{1-b^2}=\frac{1-\gamma_p^2}{1-\gamma_p^2+(\gamma_p^2-b^2)}
\le \left(1+\frac{c_2}{k^{1/4}}\right)^{-1},
\end{eqnarray*} where $c_2=\gamma_p/(1-\gamma_p^2)$.

Choosing $c=\min\{c_1, c_2\}$ from the two cases above, we get that for all $\lambda\in V_k^c$,
$$\left|\frac{\alpha^2+\rho^2}{\lambda^2+\rho^2}\right|^k\le \left(1+\frac{c}{\sqrt{k}}\right)^{-k}.$$

Thus   (\ref{basic-step-p}) is proved from which the assertion (ii) follows easily.

To prove (i) we take  $\lambda\in V_k$. Then $|\Re(\alpha-\lambda)|=|\Re \lambda|<k^{-1/4}$ and $|\Im (\alpha-\lambda)|=|\gamma_p\rho-b\rho|<k^{-1/4}$ and hence   $|\alpha-\lambda|< \sqrt{2}k^{-1/4}$. Using $|a\rho|<k^{-1/4}$ and
$0\le b\rho<\gamma_p\rho$ we get,
$$|\alpha+\lambda|^2=|a\rho+i(b\rho+\gamma_p\rho)|^2=a^2\rho^2+b^2\rho^2+\gamma_p^2\rho^2+2b\rho\gamma_p\rho
<1/\sqrt{k}+ 4\gamma_p^2\rho^2.$$
Therefore for $\lambda\in V_k$,
$|\alpha^2-\lambda^2|=|\alpha-\lambda||\alpha+\lambda|<(1/\sqrt{k}+4\gamma_p^2\rho^2)^{1/2}\sqrt{2}k^{-1/4}$  and hence  $$|\alpha^2-\lambda^2|^{4\tau} (k+N+1)(k+N+2)\dots (k+N+\tau)\le C_\tau.$$

From (\ref{explicit}) we see that   every term in $F^{k+N+1}$ contains the factor
$(\alpha^2-\lambda^2)^{N+1-\tau}=(\alpha^2-\lambda^2)^{5\tau+2}$.   The inequalities (\ref{lp-dist-1}), (\ref{explicit}) and   $|\alpha^2-\lambda^2|^{\tau+1}\le C k^{-(\tau+1)/4}$, now implies that  $F_{k+N+1}\le C k^{-(\tau+1)/4}$ where $C$ is independent of $k$.
Thus assertion (i) is proved.  It follows from (i) and (ii) that  $\sup_{\lambda\in S_p}F^k\rightarrow 0$ as $k\rightarrow \infty$ and hence
$(\alpha^2-\lambda^2)^{N+1} \what{T_0}\equiv 0$, when $T_k$ are radial. Now the argument given in Step 3 of the previous theorem will lead to the theorem restricted to   radial distributions and hence for the general case.

The proof of part (b) is also similar to that of part (b) of Theorem
\ref{thm-L^2-temp}. By the argument given there it is enough to
consider that $T_k$ are radial. If we proceed through the steps of
the proof of (a) we get, for $\phi\in C^p(\what{S})^\#$, $|\langle
\what{T_0}, \phi\rangle|\le \mu [z^k(\lambda^2+\rho^2)^{-k} \phi]$, for some
seminorm $\mu$ of $C^p(\what{S})^\#$.
Since $|z|<4\rho^2/pp' \le |\lambda^2+\rho^2|$  for all $\lambda\in
S_p$,  the right side goes to zero as $k\rightarrow \infty$. Hence
$\langle T_0, \psi\rangle=0$ for all $\psi\in C^p(S)^\#$.

(c) If $|z|> 4\rho^2/pp'$ then one can find $\theta_1$ and $\theta_2$ such that $|z|e^{i \theta_1}$ and
$|z|e^{i \theta_2}$ are in the interior of the $L^p$-spectrum. The conclusion now follows by arguing as in subsection 3.2.
\end{proof}

\section {Proof of the Theorems \ref{thm-X-2}, \ref{thm-X-p} and their analogues}
\subsection{Functions which are tempered distributions}
It follows easily from the definition of the $L^1$-Schwartz space
$C^1(S)$ that $C^1(S)\subset L^1(S)$ and hence all
$L^{\infty}$-functions on $S$ define $L^1$-tempered distribution,
i.e., $L^{\infty}(S)\subset C^1(S)'$. We shall see that the size estimates $L^{p', \infty}$,
 $M_{p'}$ and  $\mathcal A_{p', q}$    define $L^p$-tempered distributions for $1<p\le 2$.

\begin{lemma}
\label{Mp-Apq-dist}
 Let $f$ be  a measurable function on $S$ and  $1<p\le 2, 1<q<\infty$.

\noindent{\rm (a)}
There exists a seminorm
$\gamma$ of $C^p(S)$ such that for all $\phi\in C^p(S)$ and suitable functions $f$,
$$(\mathrm{i})\ \ |\langle f, \phi\rangle|\le C \|f\|_{p', \infty} \gamma(\phi), \ (\mathrm{ii})\  |\langle f, \phi\rangle|\le C M_{p'}(f) \gamma(\phi),\
(\mathrm{iii}) \ |\langle f, \phi\rangle|\le C \mathcal A_{p',q}(f) \gamma(\phi).$$

\noindent{\rm (b)} If $\lambda\in S_p$ then for each $n\in N$, $x\mapsto \wp_\lambda(x,n)$ and for each $k\in K$, $x\mapsto e^{(i\lambda+\rho)H(x^{-1}k)}$
are $L^p$-tempered distributions.
\end{lemma}
\begin{proof}

(a) We fix $p\in (1, 2]$ and $L>\max\{\rho, 1+1/2p\}$. We define a seminorm $\gamma$ on $C^p(S)$  by  $\gamma(\phi)=\sup_{x\in S}|\phi(x)|\phi_0(x)^{-2/p}(1+|x|)^{2L}$, $\phi\in C^p(S)$.
We shall first show that the radial function $h(x)= \phi_0(x)^{2/p}(1+|x|)^{-2L}$ is in $L^{p,1}(S)$, which is equivalent to showing that
$h(r)=e^{-\frac{2\rho}{p}r}(1+r)^{-2L}$ is in $L^{p, 1}([0, \infty), J(r)dr)$ where $J(r)$ is the Jacobian of the Haar measure in
polar decomposition (see (\ref{DR-polar}), (\ref{polar})).
A lengthy but routine computation  shows that the decreasing rearrangement $h^\ast$ of $h$ satisfies
\begin{equation}
h^{\ast}(t)\asymp 1, \quad t\in (0,1] \ \text{ and } \
h^{\ast}(t)=h^{\ast}(e^{2\rho u}) \asymp
e^{-\frac{2\rho}{p}u}(1+u)^{-L/\rho}\quad u\geq
0.\label{rearrangement}
\end{equation}  Since $L>\rho$ from above we  obtain
\begin{eqnarray}
\int_0^{\infty}h^{\ast}(t)t^{-1/p'}dt\nonumber &\asymp &
\int_0^1t^{-1/p'}dt+\int_1^{\infty}h^{\ast}(t)t^{-1/p'}dt\nonumber\\
&=&\int_0^1t^{-1/p'}dt+2\rho\int_0^{\infty} h^{\ast}(e^{2\rho
y})e^{-\frac{2\rho
y}{p'}}e^{2\rho y}dy\nonumber\\
&=&\int_0^1t^{-1/p'}dt+2\rho\int_0^{\infty}(1+y)^{-L/\rho}dy<\infty.\nonumber
\end{eqnarray} Thus $h\in L^{p,1}(S)$. From this (i) and (iii) follows easily. Indeed,
$$\left|\int_Sf(x)\phi(x)(x)dx\right|\leq \int_S|\phi(x)|\phi_0(x)^{-2/p}(1+|x|)^{2L}|f(x)|
h(x)dx\le \gamma(\phi)\|f\|_{p', \infty}\|h\|_{p,1} \text{ and }
$$
\begin{eqnarray*}
 |\int_Sf(x)\phi(x) dx|&\le&\gamma(\phi)\int_S
|f(x)| h(x)dx\\
&=& \gamma(\phi)\int_0^\infty \int_{\partial B(\mathfrak s)}
|f(r\omega)|h(r)d\omega\, J(r)dr\\
&\le&\gamma(\phi)\int_0^\infty \mathcal A_q(f)(a_r)
h(r)J(r)dr\\
&\le&\gamma(\phi)\mathcal A_{p',q}(f)\|h\|_{p,1}.
\end{eqnarray*}
(ii)
By definition of $M_{p'}(f)$ there exists a natural number $n_0$ such that
$\int_{B(0,R)}|f(x)|^{p'} dx \leq 2RM_{p'}(f)^{p'}$,  for all  $R\geq n_0$.
We fix one such $n_0$.
Since $\phi_0(a_r)\asymp e^{-\rho |r|}(1+|r|)$,  we have for $L$ as above
\begin{eqnarray}
&& \int_{B(0,n_0)^c}|f(x)|\frac{\phi_0(x)^{2/p}}{(1+|x|)^{2L}} dx \label{ballcomplement}\\
&\asymp&\int_{n_0}^\infty\int_{\partial B(\mathfrak s)}|f(r\omega)|\frac{e^{\frac{-2\rho}{p}r}}{(1+r)^{2L-2/p}}  e^{2\rho r} d\omega dr \nonumber\\
&\leq&\int_{n_0}^\infty \left (\int_{\partial B(\mathfrak s)}|f(r\omega)|^{p'} d\omega\right)^{1/p'}\frac{e^{\frac{2\rho}{p'} r}}{(1+r)^{2L-2/p}}dr\nonumber\\
&=&\lim_{N\rightarrow\infty}\sum_{k=0}^{N-1-n_0}\int_{n_0+k}^{n_0+k+1}\left
(\int_{\partial B(\mathfrak s)}|f(r\omega)|^{p'}
d\omega\right)^{1/p'}\frac{e^{\frac{2\rho}
{p'}r}}{(1+r)^{2L-2/p}}dr
\nonumber\\
&\leq&\lim_{N\rightarrow\infty}\sum_{k=0}^{N-1-n_0}\frac{1}{(1+n_0+k)^{2L-2/p}}
\left(\int_{n_0+k}^{n_0+k+1}\int_{\partial B(\mathfrak s)}|f(r\omega)|^{p'}e^{2\rho r}d\omega dr\right)^{1/p'}\nonumber\\
&\leq&2M_{p'}(f)\lim_{N\rightarrow\infty}\sum_{k=0}^{N-1-n_0}\frac{1}{(1+n_0+k)^{2L-1-1/p}}\le CM_{p'}(f)\nonumber
\end{eqnarray}
as $2L>2+1/p$. On the other hand \begin{equation}
\left|\int_{B(0,n_0)}f(x)\frac{\phi_0(x)^{2/p}}{(1+|x|)^{2L}}dx\right|
\leq C\left(\int_{B(0,n_0)}|f(x)|^{p'}dx\right)^{1/p'}\leq
Cn_0^{1/p'}M_{p'}(f).\label{inball} \end{equation} Combining
(\ref{ballcomplement}) and (\ref{inball})
we have   for $\phi\in C^p(S)$,
 \begin{eqnarray}
\left|\int_Sf(x)\phi(x)dx\right|&=&\left|\int_Sf(x)\frac{\phi_0(x)^{2/p}}{(1+|x|)^{2L}}
(1+|x|)^{2L}\phi_0(x)^{-2/p}\phi(x)dx\right| \leq  C M_{p'}(f)\gamma(\phi).\nonumber\end{eqnarray}

(b) We recall that for  $\lambda\in S_p$, $\phi_\lambda\in L^{p', \infty}(S)$,  $[R \wp_\lambda(\cdot, n)](x)= \wp_\lambda(e, n) \phi_\lambda(x)$
(\cite[p. 410]{ACB}) and  the radial function $h(x)=\phi_0(x)^{2/p}(1+|x|)^{-2L}\in L^{p,1}(S)$ (see (a)).  Hence   for  $\phi\in C^p(S)$ and $\lambda\in S_p$,
$$|\int_S\phi(x) \wp_\lambda(x, n) dx|\le \gamma(\phi)|\int_S h(x)\wp_\lambda(x, n) dx|=
|\wp_\lambda(e, n)| \gamma(\phi) |\int_S h(x) \phi_\lambda(x)dx|<\infty.$$
Similarly using $\int_K e^{(i\lambda+\rho)H(x^{-1}k)} dk=\phi_\lambda(x)$ we get the other assertion.
\end{proof}

\subsection{Completion of proofs} Before we enter the proofs we need to explain the statements of Theorem A and B.
Indeed it is clear that   both $\Delta$ and $\Delta^{-1}$ act as  radial $L^p$-multipliers and hence a radial $L^{p'}$-multiplier for $1<p<2$. (See \cite[Theorem 1]{Ank2}, \cite[Corollary 4.18]{ADY}.) Hence by interpolation
(\cite[p. 197]{S-W}),  they act as  radial $L^{p, \infty}$-multiplier for $1<p<\infty$. We also note that  Lemma \ref{Mp-Apq-dist} and the hypotheses of Theorem A and B ensure that    the function $f$ is  a $L^p$-tempered distribution. It follows from the definition of $C^p(S)$ that for $\phi\in C^p(S)$, both $\Delta \phi$  and $\Delta^{-1}\phi$ are functions in $C^p(S)$, where the latter is interpreted as a radial multiplier on $C^p(S)$. Hence   $\Delta^k f, k\in \Z$ can also be  considered in the sense of $L^p$-tempered distributions. i.e. $\langle \Delta^k f, \phi\rangle=\langle f, \Delta^k \phi\rangle$. Lemma \ref{Mp-Apq-dist} shows that this distributional interpretation is more robust and is applicable to the various analogues of Theorem A and B, which  will be discussed in the next subsection.
\begin{proof}[Proof of Theorem B]
Let $T_k= (4\rho^2/pp')^{-k} \Delta^k f$ for all $k\in \Z^+$.
From the hypothesis we have $(\alpha^2+\rho^2)^{-k} \Delta^k f\in L^{p', \infty}(S)$. Therefore  by Lemma \ref{Mp-Apq-dist}~(a), $T_k$ is an $L^p$-tempered distribution and
$|\langle T_k, \psi\rangle|\le M \gamma(\psi)$ for all $\psi\in C^p(S)$.
We also note that $$\Delta T_k= \left(\frac{4\rho^2}{pp'}\right)^{-k} \Delta^{k+1} f=\left(\frac{4\rho^2}{pp'}\right)\left(\frac{4\rho^2}{pp'}\right)^{-(k+1)} \Delta^{k+1} f=\left(\frac{4\rho^2}{pp'}\right)T_{k+1}.$$
Thus the sequence $\{T_k\}$ satisfies the hypothesis of Theorem \ref{thm-Lp-temp} (a) and hence $\Delta T_0=-4\rho^2/pp' T_0$, i.e. $\Delta f=-4\rho^2/pp'f$.

If $S$ is a symmetric space then by Corollary \ref{lr}, we have $f=\mathcal P_{i\gamma_{p'}\rho}F$ for some $F\in L^{p'}(K/M)$.
\end{proof}

Similarly, Theorem A follows from  Lemma \ref{Mp-Apq-dist} (a), Theorem \ref{thm-L^2-temp}  and Theorem \ref{weakl2case}. 

\subsection{Other Analogs}
It is easy to observe the following:
\begin{enumerate}
\item[(i)] Applying Lemma \ref{Mp-Apq-dist} (a) and Theorem \ref{iobs}    in Theorem~\ref{thm-L^2-temp}, we obtain a version of Theorem~A substituting  $L^{2, \infty}$-norm by $M_2$ norm.

\item[(ii)] Applying Lemma \ref{Mp-Apq-dist} (a) and Theorem  \ref{samib}  in Theorem~\ref{thm-L^2-temp}, we obtain a version of Theorem~A for hyperbolic spaces  substituting  $L^{2, \infty}$-norm by $\mathcal A_{2,q}$-norm with $1<q<\infty$.

\item[(iii)] Applying  Lemma \ref{Mp-Apq-dist} (a) and Theorem ~\ref{Mpconjecture} in Theorem~\ref{thm-Lp-temp}, we get a version of Theorem~B, substituting $L^{p', \infty}$-norm  by $M_{p'}$-norm for $1<p<2$.

\item[(iv)]    Applying Lemma \ref{Mp-Apq-dist} (a) and Theorem \ref{lohoue-knapp-p} in Theorem~\ref{thm-Lp-temp}, a version of Theorem~B is obtained where  $L^{p', \infty}$-norm is substituted by $\mathcal A_{p',q}$-norm for $1<p<2$, $1<q<\infty$.
\end{enumerate}

  Theorem~\ref{thm-L^2-temp} and  Theorem~\ref{thm-Lp-temp} and the results of section 4 also yield the following  versions of Theorem A and Theorem B.  Notice  that these theorems resemble Theorem ~\ref{bob} as $L^\infty$-norm is in use.
\begin{theorem}\label{thm-A-new-1}
For a   measurable
functions $f$ on $S$ and   $\lambda>0$, if  $\|\phi_\lambda^{-1}\Delta^k f\|_\infty\le C_\lambda (\lambda^2+\rho^2)^k$ for all $k\in \Z$,
 then $\Delta f=-(\lambda^2+\rho^2) f$.
 When $S$ is a symmetric space then,
  for $\lambda\neq 0$,  $f=\mathcal P_\lambda F$ for some $F\in L^2(K/M)$ and for
   $\lambda=0$,  $f=\mathcal P_0 F$ for some $F\in L^\infty(K/M)$.
  \end{theorem}
 \begin{proof}
It is easy to see that
for any $\phi\in C^2(S)$, $|\langle \Delta^k f, \phi\rangle|\le\gamma(\phi)\|\phi_0^{-1} \Delta^k f\|_\infty\le\gamma(\phi)\|\phi_\lambda^{-1} \Delta^k f\|_\infty$. Therefore if $T_k= (\lambda^2+\rho^2)^{-1}\Delta^k f$ then $T_k$ are $L^2$-tempered distributions and satisfies the hypothesis of  Theorem~\ref{thm-L^2-temp}. Therefore   we obtain,
$\Delta f=-(\lambda^2+\rho^2) f$.
If $\lambda\neq 0$, then $\phi_\lambda\in L^{2, \infty}(S)$ and hence  $f\in L^{2, \infty}(S)$.  Therefore for $S=X$ we apply Theorem \ref{weakl2case} to get $f=\mathcal P_\lambda F$ for some $F\in L^2(K/M)$.

If  $\lambda=0$, we apply Theorem~\ref{sos} to get $f=\mathcal P_0 F$ for some $F\in L^\infty(K/M)$.
\end{proof}
\begin{theorem}\label{thm-B-new}
For a   measurable
functions $f$ on $S$ if   $\|\phi_{i\gamma_p\rho}^{-1}\Delta^k f\|_\infty\le (4\rho^2/pp')^k$ for $k=0,1,2,\dots$,
then $\Delta f=-(4\rho^2/pp')f$.
If $S=X$ is a symmetric space, then
$f=\mathcal P_{-i\gamma_p\rho} F$ for some $F\in L^{p'}(K/M)$.
\end{theorem}
\begin{proof} As in the previous theorem, one verifies that for any $\phi\in C^p(S)$, $|\langle f, \phi\rangle|\le\gamma(\phi)\|\phi_{i\gamma_p\rho}^{-1} f(x)\|_\infty$.
Therefore   Theorem~\ref{thm-Lp-temp} can be used and we have, $\Delta f=-4\rho^2/pp' f$.  Corollary \ref{lr} now shows for $S=X$ that $f=\mathcal P_{-i\gamma_p\rho} F$
for some $F\in L^{p'}(K/M)$.
\end{proof}

\subsection{Concluding Remarks}
Let us restrict our attention to the symmetric spaces where the characterization of the Poisson  transforms is achieved.

\noindent(i)
  We fix $p\in (1,2)$. One notices that in  Theorem B and its various analogues we conclude $f=\mathcal P_{-i\gamma_p\rho}F$ for some $F\in L^{p'}(K/M)$,  while $\mathcal P_{i\gamma_p\rho}F$ for $F\in L^p(K/M)$ also satisfies the hypothesis (see (\ref{poisson-estimate-X}) and thus is a candidate to be characterized.
Indeed using the intertwining operator $I_p:L^p(K/M)\rightarrow L^{p'}(K/M)$ of Knapp and Stein (\cite{KS}), it was shown by Cowling Meda and Setti in \cite{Cow-Herz} that for $F_1\in L^p(K/M)$,
\begin{equation}
C_p \mathcal P_{i\gamma_p\rho}F_1=\mathcal P_{-i\gamma_{p}\rho} I_p(F_1),
\end{equation}
where $F=I_p(F_1)\in L^{p'}(K/M)$.  Thus $\{\mathcal P_{-i\gamma_p\rho}F\mid F\in L^{p'}(K/M)\}$ is a bigger class which includes the Poisson transforms $\mathcal P_{i\gamma_p\rho}F_1$ for $F_1\in L^p(K/M)$.

\noindent(ii) In the beginning of section 5 we have mentioned that the hypothesis of Theorem A an B (as well as their analogues) exclude the complex powers of the Poisson kernel. We have  shown that the theorems of section 5 saves the situation. For the symmetric spaces we can also weaken the hypothesis in a different way to include those complex powers of Poisson kernel. This we shall describe below. We shall  consider only the $L^2$-case, the $L^p$-case being similar. Let $\what{K}_M$ be the set of irreducible unitary representations of $K$ which contains an $M$-fixed vector. For $\delta\in \what{K}_M$, let $f_\delta$ be the $\delta$-isotypic component  of a suitable function $f$ on $X$. Precisely, $f_\delta(x)= d_\delta\chi_\delta\ast f(x)$  where  $\chi_\delta, d_\delta$  denote  respectively the dimension and trace of $\delta$. We note that if $f(x)=e^{-(i\lambda+\rho)H(x^{-1}k)}$ for some $\lambda\in \R^\times$ then  $f_\delta$  decays like  $\phi_\lambda$ (see \cite[3.11]{Ion-Pois-1}) and hence is in $L^{2, \infty}(X)$ (see section 2). In view of this we formulate the following:

\begin{theorem}
Let $f$ be a measurable function on  $X$ satisfying
for some $\alpha>0$,
 $\|\Delta^kf_{\delta}\|_{2, \infty}\le C_\delta (\alpha^2+\rho^2)^k$ for all $k\in \Z$ and $\delta\in \what{K}_M$.
Then $\Delta f=-(\alpha^2+\rho^2)f$.
\end{theorem}
Notice that the decomposition of functions and distributions in $K$-types  suggests an easier way to prove the
main results for the  symmetric spaces. We have noted this in Remark \ref{alternative-for-X}.

\noindent(iii) It should be possible to generalize some of the results considered in this paper to higher rank Riemannian symmetric spaces. We shall take it up in future.

\end{document}